\documentclass[12pt,reqno]{amsart}
\usepackage{latexsym,amsmath,amssymb,amsfonts}
\usepackage[applemac]{inputenc}
\usepackage{setspace}
\usepackage[a4paper,top=33mm, bottom=27mm, left=35mm, right=35mm]{geometry}
\usepackage{tikz}
\usepackage{graphicx}
\usepackage{color}

\usepackage{comment}

\usetikzlibrary{arrows,positioning} 
\usetikzlibrary{decorations.pathreplacing}
\usetikzlibrary{decorations.markings} 
\usepackage{amsthm}
\usepackage{amssymb}
\usepackage{mathrsfs}
\newtheorem{theorem}{Theorem}[section]
\newtheorem{corollary}[theorem]{Corollary}
\newtheorem{lemma}[theorem]{Lemma}
\newtheorem{fact}[theorem]{Fact}

\newtheorem{proposition}[theorem]{Proposition}

\theoremstyle{definition}

\newcommand{\dE}{d_E}
\newcommand{\dmax}{d_{\rm max}}
\newcommand{\m}[1]{\marginpar{\tiny{#1}}}

\newcommand{\E}{\mathbb E}
\newcommand{\R}{\mathbb R}
\newcommand{\x}{\mathbf x}
\newcommand{\y}{\mathbf y}
\newcommand{\z}{\mathbf z}
\newcommand{\X}{\mathbf X}
\newcommand{\Y}{\mathbf Y}

\newcommand{\pr}{\mathbb P}

\def\var{{\rm\bf Var}}
\newcommand{\remove}[1]{}
\newcommand{\real}{\ensuremath {\mathbb R} }
\newcommand{\far}{\rm far}

\newcommand{\eps}{\varepsilon}

\newcommand{\Bin}{{\rm Bin}}
\newcommand{\dm}[1]{{\color{blue} #1}}
\newcommand{\dmc}[1]{{\color{blue}{ \bf [~DM:\ }\emph{#1}\textbf{~]}}}
\newcommand{\RG} {\ensuremath{\mathscr G(n,r)}}

\newcommand{\SR} {\ensuremath{\mathcal S_n}}
\newcommand{\SU} {\ensuremath{\mathcal S_1}}
\def\wck#1 {\underline{#1}~\marginpar{\fbox{#1} {\tiny ?}}}
\def\silent#1\par{\par}
\def\text#1{\quad\mbox{#1}\quad}

\makeatletter
\renewcommand{\@seccntformat}[1]{\@nameuse{the#1}.\quad}
\makeatother
\newcommand{\cc}[1]{{\color{red}#1}}
\newcommand{\jdc}[1]{{\color{violet}#1}}
\begin{document}
\title{Learning random points from geometric graphs or orderings} 
\author{Josep D\'{i}az}
\address{Universitat Polit\`{e}cnica de Catalunya, Dept. de Computer Science,
08034 Barcelona}
\email{\texttt{diaz@cs.upc.edu}. Partially supported by TIN2017-86727-C2-1-R.}
\author{Colin McDiarmid}
\address{Department of Statistics, Oxford University, 24 - 29 St Giles, Oxford OX1 3LB, UK}
\email{\tt cmcd@stats.ox.ac.uk}

\author{Dieter Mitsche}
\address{Institut Camille Jordan (UMR 5208), Univ. de Lyon, Univ. Jean Monnet, 42023 Saint-Etienne, France}
\email{\texttt{dieter.mitsche@univ-st-etienne.fr}. Partially supported by IDEXLYON of Universit\'e de Lyon (Programme Investissements d'Avenir ANR16-IDEX-0005).}

\keywords {random geometric graphs, realisation problem}
\subjclass {Primary: 05C80}

\maketitle
\onehalfspace
%


\begin{abstract}
Suppose that there is a family of $n$ random points $\X_v$ for $v \in V$, independently and uniformly distributed in the square $\SR = \left[-\sqrt{n}/2,\sqrt{n}/2\right]^2$ of area $n$. We do not see these points, but learn about them in one of the following two ways.

Suppose first that we are given the corresponding random geometric graph $G\in\RG$, where distinct vertices $u$ and $v$ are adjacent when the Euclidean distance $\dE(\X_u,\X_v)$ is at most $r$.  If the threshold distance $r$ satisfies $n^{3/14} \ll r \ll n^{1/2}$,  
then the following holds with high probability.  Given the graph $G$ (without any geometric information), in polynomial time we can approximately reconstruct the hidden embedding, in the sense that, `up to symmetries', for each vertex $v$ we find a point within distance about $r$ of~$\X_v$; that is, we find an embedding with `displacement' at most about~$r$.

Now suppose that, instead of being given the graph $G$, we are given, for each vertex $v$, the ordering of the other vertices by increasing Euclidean distance from~$v$.  Then, with high probability, in polynomial time we can find an embedding with the much smaller displacement error $O(\sqrt{\log n})$.
%
%
\end{abstract}
{\small\textbf{Keywords:} Random geometric graphs, unit disk graphs, approximate embedding, vertex orders.}

\section{Introduction}\label{ap:sec:intro}

In this section, we first introduce geometric graphs and random geometric graphs, the approximate realisation problem for such graphs, and families of vertex orderings; and we then present our main theorems, give a brief sketch 
of their proofs, and finally give an outline of the rest of the paper.

\subsection{Random geometric graphs}

Suppose that we are given a non-empty finite set $V$, and an embedding $\Psi:V \to \R^2$, or equivalently a family $({\x}_v:v \in V)$ of points in $\R^2$, 
where $\Psi (v) = {\x}_v$.  Given also a real \emph{threshold distance} $r>0$, we may define
the \emph{geometric graph} $G=G(\Psi, r)$ or $G=G(({\x}_v:v \in V),r)$ with vertex set $V$ by, for each pair $u, v$ of distinct elements of $V$, letting $u$ and $v$ be adjacent if and only if  $\dE(\x_u,\x_v) \leq r$.
Here $\dE$ denotes Euclidean distance, $\dE(\x,\x')= \| \x\!-\!\x' \|_2$. Note that the (abstract) graph $G$ consists of its vertex set $V$ and its edge set (with no additional geometric information).
A graph is called \emph{geometric} if it may be written as
$G(\Psi, r)$ as above, and then $(\Psi,r)$ 
is called a \emph{realisation} of the graph. Since we may rescale so that $r=1$, a geometric graph may also be called a \emph{unit disk graph} (UDG)~\cite{Hal80}.

Given a positive integer $n$, and a real $r>0$, the \emph{random geometric graph} $G\in\RG$ with vertex set $V=[n]$ is defined as follows. Start with $n$ random points $\X_1,\ldots,\X_n$ independently and uniformly distributed in the square $\SR = \left[-\sqrt{n}/2,\sqrt{n}/2\right]^2$ of area $n$; let $\Psi(v)=\X_v$
 for each $v \in V$; and form the geometric graph $G=G(\Psi, r)$ or $G=G((\X_v\!:v \in V), r)$. 

Random geometric graphs were first introduced 
by Gilbert~\cite{Gilbert} to model communications between radio stations. Since then, several related variants of these graphs have been widely used as models for wireless communication, and have also been extensively  studied from a mathematical point of view. The main reference on random geometric graphs is the monograph by Penrose~\cite{Penrose}; see also the survey of Walters~\cite{Walters}.
The properties of $G \in \RG$ are usually investigated from an asymptotic perspective, as $n$ grows to infinity and $r=r(n)= o(\sqrt{n})$. 

A sequence $A_n$ of events holds \emph{with high probability} (whp) if $\pr(A_n) \to 1$ as $n \to \infty$. 
For example, it is well known that  $r_c=\sqrt{\log n/\pi}$ is a sharp threshold function for the connectivity of the random geometric graph $G \in \RG$.
This means that, for every $\varepsilon>0$, if $r \le (1-\varepsilon) r_c$, then $G$ is whp disconnected, whilst if
$r \ge (1+\varepsilon) r_c$, then $G$ is whp connected 
(see~\cite{Penrose} for a more precise result).
We shall work with much larger~$r$, so our random graphs will whp be (highly) connected.

Given a graph $G$, we define the {\em graph distance} $d_G(u,v)$ between two vertices $u$ and $v$ to be the least number of edges in a path between $u$ and $v$
if $u$ and $v$ are in the same component, and if not then we let the distance be $\infty$.  Observe that in a geometric graph $G$ with a given realisation $(\Psi, r)$, each pair of vertices $u$ and $v$ must satisfy $d_G(u,v) \geq \dE(\Psi(u), \Psi(v))/r$, since each edge of the embedded geometric graph has length at most $r$. For a finite simple graph~$G$ with $n$ vertices, 
let $A(G)$
denote its adjacency matrix, the $n \times n$ symmetric matrix with 
$i,j$ entry 1 if $i j$ is an edge and 
 0 otherwise.  (We write $ij$ for an edge rather than the longer form $\{i,j\}$.)

\subsection{Approximate realisation for geometric graphs}

For a geometric graph $G$ with vertex set $V$, the \emph{realisation problem} for $G$ has as input the adjacency matrix $A(G)$, and consists in finding some realisation $(\Psi,r)$.
It is known that for UD graphs, the realisation problem (also called the unit disk graph reconstruction problem) is NP-hard~\cite{Breu98}, and it remains NP-hard even if we are given all the distances between pairs of vertices in some realisation~\cite{Aspnes04}, or if we are given all the angles between incident edges in some realisation~\cite{Bruck05}.
%
Given that these results indicate the difficulty in finding exact polynomial time algorithms, researchers naturally turned their attention to 
finding {good} {\em approximate} realisations (for deterministic problems). \\

\noindent
\emph{Previous work on approximate realisation}

There are different possible measures of `goodness' of an embedding. Motivated by the localisation problem for 
sensor networks, see for example~\cite{Bulusu04}, (essentially) the following scale-invariant measure of quality of embedding was introduced in~\cite{Moscibroda04}: given a geometric graph $G=(V,E)$, and an embedding $\Phi : V \to \R^2$ and threshold distance $r>0$, if $G$ is not a clique we let
$$Q_G(\Phi)=\frac{\max_{xy \in E} \|\Phi(x)- \Phi(y) \|_2}{\min_{ xy \notin E} \|\Phi(x)- \Phi(y) \|_2}$$
(where we insist that $x \neq y$); and let $Q_G(\Phi)=(1/r) \max_{xy \in E} \|\Phi(x)- \Phi(y) \|_2$ if $G$ is a clique.  Observe that if $(\Phi, r)$ is a realisation of $G$ then $Q_G(\Phi) \leq 1$. 
 The aim is to find an embedding $\Phi:V\to \R^2$ with say $r=1$ which minimises $Q_G(\Phi)$, or at least makes it small. 
The random projection method~\cite{Vempala98} was used in~\cite{Moscibroda04} to give an algorithm that, for an $n$-vertex UD graph $G$, outputs an embedding $\Phi$ with $Q_G(\Phi) = O(\log^{3.5}n \, \sqrt{\log\log n})$; that
is, it
approximates feasibility in terms of the measure $Q_G$ up to 
 a factor of $O(\log^{3.5}n \, \sqrt{\log\log n})$. 
On the other hand, regarding inapproximability, it was shown in~\cite{Kuhn04} that, for any $\eps > 0$,
it is NP-hard for UD graphs
to compute an embedding $\Phi$ with $Q_G(\Phi) \leq \sqrt{3/2}-\!\eps$. 

In this paper we do not aim to control a goodness measure like~$Q$ (though see the discussion following Theorem~\ref{thm:main2}).  Instead, we find whp a `good' embedding $\Phi$, which is `close' to the hidden original random embedding $\Psi$.
 We investigate the approximate realisation problem for a random geometric graph, and for a family of vertex orderings (see later).
 
What we achieve for random geometric graphs is roughly as follows.
We describe a polynomial time algorithm which, for a suitable range of values for~$r$, whp finds an embedding $\Phi$ which `up to symmetries' (see below for a detailed definition) maps each vertex $v$ to within about distance $r$ of the original random point $\Psi(v) = \X_v$.
Observe 
that the mapping $\Phi$ must then satisfy the following properties whp: for each pair of vertices $u,v$ with $\dE(\Psi(u),\Psi(v)) \leq r$
 we have $\dE(\Phi(u),\Phi(v)) \leq (3+\eps)r$, and for each pair of vertices $u,v$ with 
$\dE(\Psi(u),\Psi(v)) \ge (3+ \eps)r$  we have $\dE(\Phi(u),\Phi(v)) > r$. 
Thus, adjacent pairs of vertices remain quite close to being adjacent in $\Phi$, and non-adjacent pairs of vertices that are sufficiently far apart remain non-adjacent in $\Phi$.

For maps $\Phi_1, \Phi_2 : V \to \SR$, the familiar \emph{max} or \emph{sup distance} is defined by
\[ \dmax(\Phi_1,\Phi_2) = \max_{v \in V} \dE(\Phi_1(v), \Phi_2(v)).\]
Since there is no way for us to distinguish embeddings which are equivalent up to symmetries, we cannot hope to find an embedding $\Phi$ such that whp $\dmax(\Psi,\Phi)$ is small.
There are 8 symmetries (rotations or reflections) of the square. We define the \emph{symmetry-adjusted sup distance} $d^*$  by
\[ d^*(\Phi_1,\Phi_2) = \min_{\sigma} \dmax(\sigma\!\circ\!\Phi_1,\Phi_2) = \min_{\sigma} \dmax(\Phi_1,\sigma\!\circ\!\Phi_2),\]
where the minima are over the 8 symmetries $\sigma$ of the square $\SR$.  This is the natural way of measuring distance `up to symmetries'.
If we let $\Phi_1 \sim \Phi_2$ when $\Phi_1= \sigma\!\circ\!\Phi_2$ for some symmetry $\sigma$ of $\SR$, then it is easy to check that $\sim$ is an equivalence relation on the set of embeddings $\Phi: V \to \SR$, 
and $d^*$ is the natural sup metric on the set of equivalence classes.

Given $\alpha>0$, we say that an embedding $\Phi$ has \emph{displacement at most} $\alpha$ (from the hidden embedding $\Psi$)
if $d^*(\Psi,\Phi) \leq \alpha $.
Consider the graph with three vertices $u, v, w$ and exactly two edges $uv$ and $vw$:
 if this is the geometric graph $G(\Psi,r)$, then $\dE(\Psi(u),\Psi(w))$ could be any value in $(r,2r]$.  Examples like this suggest that we should be happy to find an embedding $\Phi$ with displacement at most about $r$;
and since our methods rely on graph distances, it is natural that we do not achieve displacement below $r$. 

\subsection{Vertex orderings}
\label{subsec.vorders}

We also consider a related approximate realisation problem, with different information.  As for a random geometric graph, we start with a family of $n$ unseen points $\X_1, \ldots, \X_n$ independently and uniformly distributed in the square $\SR$, forming the hidden embedding $\Psi$. (There is no radius $r$ here, and there is no graph.)
We are given, for each vertex $v$, the ordering $\tau_v$ 
 of the 
vertices by increasing Euclidean distance from~$v$.  This is the \emph{family of vertex orderings} corresponding to $\Psi$.
%
 Notice that with probability~$1$ no two distances between distinct points will be equal. 
Notice also that, if we had  access to the complete ordering of the Euclidean distances between all pairs of distinct vertices in the hidden embedding $\Psi$, then we could read off the family of vertex orderings.
 
We shall see that, by using the family of vertex orderings, we can with high probability find an embedding with displacement error dramatically better than the bound we obtain for random geometric graphs.

\subsection{Main results}
%
Suppose first that we are given a random geometric graph $G\in\RG$, with hidden original embedding  $\Psi$, for example by being given the adjacency matrix $A(G)$. We of course know the number $n$ of vertices and the number of edges, but we do not know the threshold distance $r$, and indeed we have no geometric information.
Our goal is to find an embedding $\Phi$ such that whp it has displacement at most about $r$, for as wide as possible a range of values for~$r$.
However, first we need to consider how to estimate~$r$. 
We determine the expected number of edges in $G\in\RG$ as a function of $n$ and $r$ (taking care over boundary effects), see Proposition~\ref{prop.M};
and use this value and the observed number of edges to calculate the estimator $\hat{r}$ of $r$.

\begin{proposition} \label{prop.rhat}
Let $r=r(n)>0$ satisfy $1/\sqrt{n} \ll r \ll \sqrt{n}$ as $n \to \infty$.
Let $\rho= \sqrt{n}/r$ (so $\rho \to \infty$ as $n \to \infty$).
Let $\omega=\omega(n)$ be a function tending to infinity with~$n$ arbitrarily slowly, and in particular satisfying $\omega \ll \sqrt{n}$.
Then in $O(n^2)$ time
we can compute an estimator $\hat{r}$ such that 
\begin{equation} \label{eqn.rhatN}
 |\hat{r} -r| \ll \;
 \omega \cdot \,(n^{-1/2} + \rho^{-3/2}) \; \mbox{ whp},
\end{equation}
and thus 
$\hat{r}/r \to 1$ in probability as $n \to \infty$.
\end{proposition}


Our first theorem presents an algorithm to find an embedding $\Phi$ for a random geometric graph 
(given without any further information), which whp achieves displacement at most about $r$, for the range $n^{3/14} \ll r \ll \sqrt{n}$.  Note that $3/14 \approx 0.21428$.
%
\begin{theorem}\label{thm:main}
Let $r=r(n)$ satisfy $\: n^{3/14} \ll r \ll \sqrt{n}$, and consider the random geometric graph $G\in\RG$
(given say by the adjacency matrix $A(G)$)
corresponding to the hidden embedding $\Psi$. 
Let $\eps > 0$ be an arbitrarily small constant. There is an algorithm which in $O(n^2)$ time 
outputs an embedding $\Phi$ which whp has displacement at most $(1+\eps)r$, that is, whp $d^*(\Psi,\Phi) \leq (1+\eps)r$.
\end{theorem}
%

In practice, after running the algorithm in this theorem, we would run a local improvement heuristic, even though this might 
not lead to a provable decrease in $d^*(\Psi,\Phi)$.  For example, we might simulate a dynamical system where, for each vertex $v$, the corresponding point $\x_v$ 
tends to move towards the centre of gravity of the points $\x_w$ corresponding to the neighbours $w$ of $v$ (if $\x_v$ is not too close to the boundary of $\SR$).
\smallskip

Our second theorem concerns the case when we learn not the random geometric graph but the family of vertex orderings; that is, for each vertex $v$, we learn the ordering $\tau_v$ of the vertices by increasing Euclidean distance from~$v$, given as a linked list.
%
\begin{theorem}\label{thm:main2}
Suppose that we are given the family 
of vertex orderings corresponding to the hidden embedding $\Psi$.
There is a linear time algorithm that
outputs an embedding $\Phi$ which whp has displacement 
$< 1.197 \sqrt{\log n}$;
 that is, whp $d^*(\Psi,\Phi) < 1.197 \sqrt{\log n}$. 
\end{theorem}
Now let $\sqrt{\log n} \ll r \ll \sqrt{n}$, and let $G\in\RG$ be the random geometric graph corresponding to the hidden embedding $\Psi$.
Then the constructed embedding $\Phi$ does well in terms of the measure $Q_G$ introduced earlier: we have
\begin{equation} \label{eqn.Qgood}
 Q_G(\Phi) < \frac{r+ 1.197 \sqrt{\log n}}{r- 1.197 \sqrt{\log n}} < 1 + 2.4 \sqrt{\log n}/r = 1+o(1) \;\; \mbox{ whp}.
\end{equation}
Also, from the constructed embedding $\Phi$ we may form a second geometric graph $G'=G(\Phi, r)$.
Then $G'$ is close to $G$ in the sense that `we get only a small proportion of edges wrong'.
We make this more precise in the inequality~(\ref{eqn.symm}) below.  It is easy to see that whp $G$ has $(\frac12\!+\!o(1)) \pi r^2 n$ edges (see Proposition~\ref{prop.M} for a more detailed result). 
We know from Theorem~\ref{thm:main2} that whp $\Phi$ has displacement 
 $< 1.197 \sqrt{\log n}$: 
 assume that this event holds. If $\dE(\X_u, \X_v) \leq r - 2.394 \sqrt{\log n}$ then $\dE(\Phi(u), \Phi(v)) \leq r$ so $uv$ is an edge in $G'$; and similarly, if
  $\dE(\X_u, \X_v)  \geq  r + 2.394 \sqrt{\log n}$ then $\dE(\Phi(u), \Phi(v)) > r$, so $uv$ is not an edge in $G'$.
Thus there could be a mistake with $uv$ only if
\[ r - 2.394 \sqrt{\log n} < \dE(\X_u, \X_v)  <  r + 2.394 \sqrt{\log n}.\]
But whp the number of unordered pairs $\{u,v\}$ of distinct vertices such that these inequalities hold is $(1\!+\!o(1))\, n \cdot 2 \pi r  \cdot 2.394 \sqrt{\log n}$. 
 Hence, whp the symmetric difference of the edge sets of $G$ and $G'$ satisfies
\begin{equation} \label{eqn.symm}
\big| E(G) \Delta E(G') \big|\, / \, \big| E(G) \big| < 9.6 \sqrt{\log n} /r. 
\end{equation}
\smallskip

\noindent
\emph{Outline sketch of the proofs of Theorems~\ref{thm:main} and~\ref{thm:main2}}

In order to prove these theorems, we first identify $4$ `corner vertices' such that the corresponding points are close to the $4$ corners of $\SR$.  To do this,
for Theorem~\ref{thm:main} we are guided by vertex degrees; and for Theorem~\ref{thm:main2} we look at the set of `extreme' pairs $\{v, v'\}$ such that $v'$ is farthest from $v$ in the order $\tau_v$, and $v$ is farthest from $v'$ in the order $\tau_{v'}$.

To prove Theorem~\ref{thm:main}, we continue as follows. For a vertex $v$, we approximate the Euclidean distance between $\X_v$ and a corner by using the graph distance from $v$ to the corresponding corner vertex, together with the estimate $\hat{r}$ of $r$; and then we place our estimate $\Phi(v)$ for $\X_v$ at the intersection of circles of appropriate radius 
centred on a chosen pair of the corners.
For each of the circles, whp $\X_v$ lies within a narrow annulus around it, so $\Phi(v)$ is close to $\X_v$.

In the proof of Theorem~\ref{thm:main2}, 
we obtain a much better approximation to the Euclidean distance between $\X_v$ and a corner, by using the rank of $v$ in the ordering from the corresponding corner vertex, and the fact that the number of points $\X_w$ at most a given distance from a given corner is concentrated around its mean.  In this way, we obtain much narrower annuli, and a correspondingly much better estimate for $\X_v$.

\smallskip

\subsection{Further related work}\label{subs:further} 
In this section we mention further related work.

Theorem 1 of~\cite{Arias} estimates Euclidean distances between points by $r$ times the graph distance in the corresponding geometric graph. It is assumed that $r$ is known, and the error is at most $r$ plus a term involving the maximum radius of an empty ball. In the case of $n$ points distributed uniformly and independently in $\SR$,  the authors of~\cite{Arias} need $r \ge n^{1/4}(\log n)^{1/4}$ in order to keep the error bound down to $(1+o(1))r$ whp (so they need $r$ a little 
 larger than we do in Theorem~\ref{thm:main}).

In~\cite{Partha} the authors assume that they are given a slightly perturbed adjacency matrix (some edges were inserted, some were deleted) of $n$ points in some metric space. Under fairly general conditions on insertion and deletion, the authors use the Jaccard index (the size of the intersection of the neighbourhood sets of the endpoints of an edge divided by the size of their union) to compute a $2$-approximation to the graph distances.

The use of graph distances for predicting links in a dynamic social network such as a co-authorship network was experimentally analysed in~\cite{Liben2}: it was shown that 
graph distances (and other approaches) can provide useful information to predict the evolution of such a network. In~\cite{Sarkar} the authors consider a deterministic and also a non-deterministic model, and show that using graph distances, and also using common neighbours, they are able to predict links in a social network. The use of shortest paths in graphs for embedding points was also experimentally analysed in~\cite{Shang}. See~\cite{KruskalSeery} for an approach using graph distances and multidimensional scaling to assign geometric locations to vertices (in such a way that the Euclidean distances between the points "best" match the graph distances), 
with various applications, in particular to graph drawing and graph realisation.  

In~\cite{Luxburg} the authors consider a $k$-nearest neighbour graph on $n$ points $\X_i$ that have been sampled iid from some unknown density in Euclidean space. They show how shortest paths in the graph can be used to estimate the unknown density. In~\cite{Terada} the authors consider the following problem: given a set of indices $(i,j,k,\ell)$, 
 together with constraints $d_E(\X_i,\X_j) < d_E(\X_k,\X_\ell)$ (without knowing the distances), construct a point configuration that preserves these constraints as well as possible. The authors propose a `soft embedding' algorithm which not only counts the number of violated constraints, but takes into account also the amount of violation of each constraint. Furthermore, the authors also provide an algorithm for reconstructing points when only knowing the $k$ nearest neighbours of each data point, and they show that the obtained embedding converges for $n \to \infty$ to the real embedding (w.r.t. to a metric defined by the authors), as long as $k \gg \sqrt{n \log n}$. This setup is similar to our Theorem~\ref{thm:main2} in the sense that we are given the ordinal ranking of all distances from a point (for each point), though note that we estimate points up to an error $O(\sqrt{\log n})$ rather than $o(\sqrt{n})$ (recall that our points are sampled from the $\sqrt{n} \times \sqrt{n}$ square $\SR$). See~\cite{Bernsteinetal} for a similar approach to detecting a $d$-dimensional manifold in an $N$-dimensional space with $d \ll N$ based on sufficiently dense point sampling, using approximations of geodesic distances by graph distances.

In a slightly different context,  the algorithmic problem of computing the embedding of $n$ points in Euclidean space given some or all pairwise distances was considered. 
If all $\binom{n}{2}$ pairwise distances are known, and we are given the positions of three points forming a triangle $T$, then we can easily find exact positions in $O(n)$ arithmetic operations:
separately for each point $\x$ not in $T$, find its location with respect to $T$ by intersecting the three circles corresponding to the distances from $\x$ to the points in $T$, using $O(1)$ arithmetic operations. In this way we use only the $O(n)$ distances involving at least one of the points in $T$.
%
In~\cite{Cucuringu2, Cucuringu} the authors consider the problem of knowing only a subset of the distances (they know only small distances, as typical in sensor networks), and show that by patching together local embeddings of small subgraphs an 
approximate embedding of the points can be found quickly (in polynomial time).

The related problem of trying to detect latent information on communities in a geometric framework was studied by~\cite{Baccelli}. In this case, (visible)
points of a Poisson process in the unit square are equipped with an additional hidden label indicating to which of two hidden communities they belong. 
The probability that two vertices are joined by an edge naturally depends on the distance between them, but also edges between vertices of the same label have a higher probability to be present than edges between vertices of different labels. The paper gives exact recovery results for a dense case, and also shows the impossibility of recovery in a sparse case.

\subsection{Organisation of the paper}
In Section~\ref{sec:Prelim} we recall or establish preliminaries; in Section~\ref{sec.randdErev} we investigate the distribution of the number of edges in $G\in\RG$, use the number of edges to give an estimator $\hat{r}$ for the threshold distance $r$,  and estimate Euclidean distances using $\hat{r}$ and graph distances; in Section~\ref{sec:Main} we complete the proof of Theorem~\ref{thm:main}; in Section~\ref{sec:Main2new} we prove Theorem~\ref{thm:main2}; and in Section~\ref{sec:Conclusion} we conclude with some open questions.


\section{Preliminaries}\label{sec:Prelim}

In this section we gather simple facts and lemmas that are used in the proofs of the main results. We start with 
a standard version of the Chernoff bounds for binomial random variables, see for example Theorem 2.1 and inequality (2.9) in~\cite{JLR}. 
%
\begin{lemma}  (Chernoff bounds) \label{lem:Chernoff}
Let $X$ have the binomial distribution $\Bin(n,p)$ 
with mean $\mu=np$. For every $\delta>0$ we have
\[ \pr(X \leq (1-\delta) \mu) \leq e^{-\delta^2 \mu/2}\]
and
\[ \pr(X \geq (1+\delta) \mu) \leq 
e^{-\delta^2(1+\delta/3)^{-1} \mu/2};\]
and it follows that, for each $0 < \delta \le 1$,
\[ \pr(|X-\mu| \geq \delta \mu) \leq 2 e^{-\delta^2 \mu/3}.\]
\end{lemma}


%
\smallskip

For $\x \in \R^2$ and $r>0$, let $B(\x,r)$
denote the closed ball of radius $r$ around~$\x$. We shall repeatedly use the following elementary fact.

\begin{fact} \label{factBinom}
Let $G\in\RG$ be a random geometric graph.  For each $\x \in \SR$ let $\sigma_n(\x)$
 be the area of $B(\x,r) \cap \SR$, and let $\rho_n(\x)= \sigma_n(\x)/n$.
Then for each vertex $v \in V=[n]$ and each point $\x \in \SR$, $\deg_G(v)$ conditional on $\X_v=\x$ has distribution $\Bin(n\!-\!1, \rho_n(\x))$.  More precisely, this gives a density function: for any Borel set $A \subseteq \SR$, 
\[ \pr((\deg_G(v)=k) \land (\X_v \in A)) = \int_{\x \in A}  \pr(\Bin(n\!-\!1, \rho_n(\x))=k) \, d\x.\]
In particular, if $\, \rho^- \leq \rho_n(\x) \leq \rho^+$ for each $\x \in A$, then, conditional on $\X_v \in A$,  $\deg_G(v)$ is stochastically at least $\Bin(n\!-\!1,\rho^-)$ 
and stochastically at most $\Bin(n\!-\!1,\rho^+)$.
\end{fact}
\smallskip

The next lemma gives elementary bounds on the area $\sigma_n(\z)$ for $\z \in \SR$, in terms of the distance from $\z$ to a corner of $\SR$ or to the boundary of $\SR$.

\begin{lemma} \label{lem.extraarea}
Let $0<s \leq r< \sqrt{n}/2$, and let $\z \in \SR$.
\begin{enumerate}
\item 
If $\z$ is at distance at most $s$ from some corner, then $\sigma_n(\z) \leq \tfrac14 \pi (r+s)^2$.
\item If $\z$ is at distance at least $s$ from each corner, then $\sigma_n(\z) \geq \tfrac14 \pi r^2 + s(r-s/2)$.
\item If $\z$ is at distance at most $s$ from the boundary, then $\sigma_n(\z) \leq \tfrac12 \pi r^2 + 2sr$.
\item If $\z$ is at distance at least $s$ from the boundary and at distance at most $r$ from at most one side of the boundary, then
$\sigma_n(\z) \geq \tfrac12 \pi r^2 + 2s(r-s/2)$.
\end{enumerate}
\end{lemma}
\begin{proof}
Parts (i) and (iii) are easy. 
 To prove parts (ii) and (iv), we  
observe first that, in the disk with centre $(0,0)$ and radius $r$, the set $S$ of points $(x,y)$ in the disk with $-s \leq x \leq 0$ and $y \geq 0$ has area at least $s(r-s/2)$.  For if $(-s,y_1)$ is the point on the bounding circle with $y_1>0$, then $y_1=\sqrt{r^2-s^2} \geq r-s$, so the quadrilateral $Q$ with corners $(0,0), (-s,0), (-s,y_1)$ and $(0,r)$ has area $\geq \frac12 s(r+r-s)$, and $Q \subseteq S$.

\begin{figure}\label{fig.extraarea}
\scalebox{0.85}{%
\begin{tikzpicture}
 \coordinate (A) at (0,0);
\coordinate (B) at (2,1.5);
 \coordinate (C) at (0,6);
 \coordinate (D) at (7,0);
\coordinate (E) at (2.5, 0);
\coordinate (F) at (0,2.5);
\coordinate (G) at (2,7);
 \coordinate (H) at (6,1.5);
\coordinate (J) at (2,0);
\coordinate (K) at (0,1.5);
\fill[red]  (B) circle [radius=0.1];
\draw [thick] (C)--(A)--(D);
\draw [black] (B)--(5.7,0);
\draw [blue,dashed,thick] (K)--(B);
\draw [blue,dashed,thick] (J)--(B);
\draw[black] (B)--(H);
\draw [red,thin] (A)--(B);
\draw     [red] (2.5,0) arc[start angle=0, end angle=90, radius=2.5];
\draw [black] (5.7,0) arc[start angle=-21.5, end angle=118.5, radius=4.04];
\draw[red,dashed] (1,0) arc[start angle=0, end angle=70, radius=.5];
\node [left] at (A) {{\footnotesize{$(-\sqrt{n}/2, -\sqrt{n}/2)$}}};
\node [right,red] at (2,1.7) {${\bf z}=(x,y)$};
\node [blue,left ] at (2,0.7) {$y$};
\node [above] at (1,0.7) {$s$};
\node [below] at (4.1,0.7) {$r$};
\node [below] at (4.2,1.5) {$r$};
\node [right] at (1,0.3) {$\theta$};
\node [below,] at (1.1,0) {$s$};
\draw [dashed,<-] (0,-0.2)--(1.1,-0.2);
\draw [dashed,->] (1.35,-0.2)--(2.5,-0.2);
\node [blue,above] at (1,1.5) {$x$};
\end{tikzpicture}}
\caption{Figure for the proof of Lemma~\ref{lem.extraarea}(ii)}\label{fig10}
\end{figure}
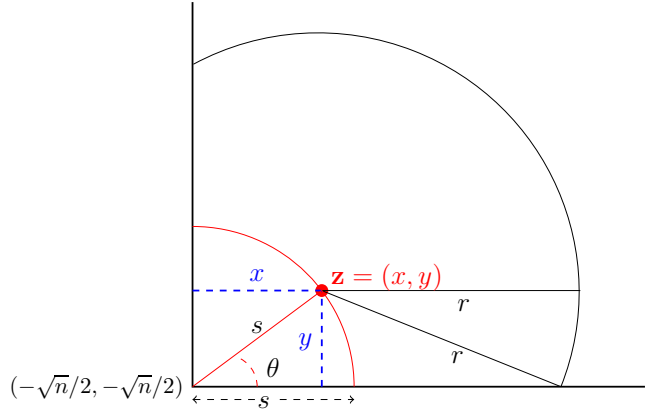
To prove part (ii) of the lemma (see Figure~\ref{fig10}), 
it suffices to consider points $\z \in \SR$ at distance equal to $s$ from a corner, wlog from the bottom left corner $c_1= (-\sqrt{n}/2, -\sqrt{n}/2)$.
%
Suppose that $\z -c_1 = (x,y)$, so $x,y \ge 0$.  Then, by the observation in the first paragraph, 
$$
  \sigma_n(\z) - \tfrac14 \pi r^2 \geq xy+ x(r-x/2) + y(r-y/2) \geq x(r-x/2)+y(r-y/2).
  $$
  Now, set $x= s \cos \theta$, $y=s \sin \theta$ for some $0 \le \theta \le \pi/2$, and the right hand side of the previous inequality can be written as
  \begin{eqnarray*}
  f(\theta) &:=& s \cos \theta (r - \tfrac12 s \cos \theta)+ s \sin \theta(r - \tfrac12 s \sin \theta) \\
  &=& sr (\cos \theta + \sin \theta) - \tfrac12 s^2 \cos^2 \theta - \tfrac12 s^2 \sin^2 \theta \\
  &=& sr(\cos \theta + \sin \theta) - \tfrac12 s^2. 
  \end{eqnarray*}
We have $f''(\theta) = - sr(\cos \theta + \sin \theta) < 0$
for $0 < \theta < \pi/4$, 
so $f$ is concave, and thus $f(\theta)$ is minimised over $0 \leq \theta \leq \pi/4$ at $\theta=0$ or $\theta=\pi/2$. But $f(0)=f(\pi/2)=sr-\frac12s^2 = s(r- s/2)$, and part (ii) follows. 
Part (iv) follows similarly from the initial observation. 
\end{proof}

We shall depend heavily on the following 
result on the relation between graph distance and Euclidean distance for random geometric graphs (with slightly worse constants than the ones given in the original paper to make the expression cleaner).
\begin{lemma}\cite{DMPP14+}[Theorem 1.1]\label{Thm:DMPP}
Let $G\in\RG$ be a random geometric graph with $r \gg \sqrt{\log n}$. Then, whp, for every pair of vertices $u,v$ we have:
$$ d_G(u,v) \leq \left\lceil \frac{d_E(\X_u, \X_v)}{r}\left(1+\gamma\, r^{-4/3}\right) \right\rceil$$
where
$$\gamma=\max\left\{3000 \left(\frac{r\log{n}}{r+d_E(\X_u,\X_v)}\right)^{2/3},\;\frac{4 \cdot 10^6\log^2{n} }{r^{8/3}},\; 1000\right\} .$$
\end{lemma}
We observed earlier that always $d_E(\x_u,\x_v) \leq r d_G(u,v)$; we next give a corollary of the last lemma which shows that whp this bound is quite tight.
\begin{corollary} \label{cor.dists}
There is a constant $c$  ($\, \leq 6\cdot10^6$) 
such that, if $r \geq (\log n)^{3/4}$ 
 for $n$ sufficiently large, then whp, for every pair of vertices $u,v$ we have: 
$$
d_G(u,v) \le d_E(\X_u,\X_v)/r + 1 + c \max \{ n^{1/2}r^{-7/3}, n^{1/6} r^{-5/3} (\log n)^{2/3} \}.
$$
\end{corollary}
\begin{proof}
By Lemma~\ref{Thm:DMPP}
\begin{equation} \label{eqn.rdG}
r d_G(u,v) \leq d_E(\X_u,\X_v) + r + r \cdot \dE(\X_u,\X_v) \gamma \, r^{-7/3}.
\end{equation}
But, for $r \geq (\log n)^{3/4}$, 
the second term in the maximum in the definition of $\gamma$ is at most $4 \cdot 10^6$; and letting $\gamma_1$ denote the first term we have 
\begin{eqnarray*}
d_E(\X_u,\X_v)\, \gamma_1 \, r^{-7/3} & \leq &
3000 \, d_E(\X_u,\X_v)^{1/3} (r \log n)^{2/3} r^{-7/3}\\
& \leq &
3000 \, (2n)^{1/6} r^{-5/3} (\log n)^{2/3}.
\end{eqnarray*}
Thus
\begin{eqnarray*}
&& d_E(\X_u,\X_v)\, \gamma \, r^{-7/3}\\
& \leq &
\max\{3000 \, (2n)^{1/6} r^{-5/3} (\log n)^{2/3}, (4 \cdot 10^6) (2n)^{1/2} r^{-7/3}\}
\end{eqnarray*}
and the lemma follows from~(\ref{eqn.rdG}).
\end{proof}
In fact, all we shall need from the last two results is the following immediate consequence of the last one.
\begin{corollary} \label{cor.dists2}
If $r \gg n^{3/14}$, then there exists $\eps=\eps(n) = o(1)$ such that whp, for every pair $u,v$ of vertices, we have 
\[d_G(u,v) \leq d_E(\X_u,\X_v)/r + 1 + \eps.\]
\end{corollary}
\smallskip

We consider the four corner points $c_i=c_i(n)$ of $\SR$ in clockwise order from the bottom left: $c_1= (-\sqrt{n}/2, -\sqrt{n}/2)$ (already defined), $c_2 = (-\sqrt{n}/2, \sqrt{n}/2)$, $c_3 = (\sqrt{n}/2, \sqrt{n}/2)$ and $c_4 = (\sqrt{n}/2, -\sqrt{n}/2)$.  See Figure~\ref{fig:corner} for the points $c_i$ and to illustrate the following lemma.

\begin{lemma} \label{lem.corners}
Let $r=r(n)$ satisfy $\sqrt{\log n} \ll r \ll \sqrt{n}$ 
and consider the random geometric graph $G \in \RG$. 
Let $\omega=\omega(n)$ tend to infinity with $n$ arbitrarily slowly, and in particular assume that $\omega^2 \leq r/2$ and $\omega \ll r/\sqrt{\log n}$.
Then whp the following holds:
(a) for each $i=1,\ldots,4$, there exists $v_i \in V$ such that $\X_{v_i} \in B(c_i,\omega)$ and $\deg_G(v_i) < \tfrac14 \pi r^2+ \tfrac13 \omega r$;
 and (b) for each $v \in V$ such that $\X_v \not\in \cup_{i=1}^4 B(c_i,\omega)$ we have $\deg_G(v) > \tfrac14 \pi r^2+ \tfrac12 \omega r$.
\end{lemma}
\begin{proof}
(a) Fix $i \in [4]$.  Note first that
\[ \pr(\X_v \not\in B(c_i,\tfrac17 {\omega}) \mbox{ for each } v \in V) = (1- \tfrac{\pi}{4n} (\tfrac{\omega}{7})^2)^n \leq e^{- \tfrac{\pi}{196} \omega^2} = o(1);\]
so whp there exists $v_i \in V$ such that $\X_{v_i} \in B(c_i,\frac17 \omega)$. 
Let $Z_n^{(i)}$ be the number of vertices $v$ such that $\X_v \in B(c_i,\tfrac1{7} \omega)$.  Then $\E[Z_n^{(i)}] = \tfrac{\pi}{196} \omega^2$.
For each $\x \in B(c_i,\frac1{7} \omega)$, by Lemma~\ref{lem.extraarea} (i),
\[\sigma_n(\x) \leq \tfrac14 \pi(r+\tfrac1{7} \omega)^2 \leq \tfrac14 \pi r^2 + \tfrac14 \omega r\]
for $n$ sufficiently large; and then, by Lemma~\ref{lem:Chernoff} and Fact~\ref{factBinom}, for each vertex~$v$
\begin{eqnarray*}
&&
\pr(\deg_G(v) \geq \tfrac14 \pi r^2+ \tfrac13 \omega r\mid \X_v \in B(c_i,\tfrac1{7} \omega))\\
& \leq &
\pr( \Bin(n,(\tfrac14 \pi r^2 + \tfrac14 \omega r)/n) \, \geq \, (\tfrac14 \pi r^2 + \tfrac14 \omega r) + \tfrac1{12} \omega r) \; \leq \;
e^{-\Theta(\omega^2)}.
\end{eqnarray*}
Hence
\begin{eqnarray*}
&& \pr(\mbox{for some } v \in V, \, (\X_v \in B(c_i,\tfrac1{7} \omega)) \land (\deg_G(v) \geq \tfrac14 \pi r^2 + \tfrac13 \omega r))\\
& \leq &
 \sum_{v \in V} \pr(\X_v \in B(c_i,\tfrac1{7} \omega))\, \pr(\deg_G(v) \geq \tfrac14 \pi r^2 + \tfrac13 \omega r \mid \X_v \in B(c_i,\tfrac1{7} \omega))\\
 & \leq &
 \E[Z_n^{(i)}] \, e^{-\Theta(\omega^2)} \; = \; o(1).
\end{eqnarray*} 
Thus whp there exists $v_i \in V$ such that $\X_{v_i} \in B(c_i,\tfrac1{7} \omega)$ and $\deg_G(v_i) < \tfrac14 \pi r^2 + \tfrac13 \omega r$.
This gives part (a) of the lemma.\\

(b) 
Let $j_0= \lfloor \omega \rfloor$ and $j_1= \lceil r / \omega \rceil$. 
For all integers $i \in [4]$ 
and $ j_0 \leq j \leq j_1$, let $B_i^j= B(c_i,j) \cap \SR$.  Consider first the central part of the square $\SR$, omitting parts near the corners: let $C_n =\SR \setminus \cup_i B_i^{j_1}$.  By Lemma~\ref{lem.extraarea} (ii), for each $\x \in C_n$ we have 
\[\sigma_n(\x) \geq \tfrac14 \pi r^2 +j_1(r-j_1/2) \geq \tfrac14 \pi r^2 + \tfrac12 j_1 r \]
for $n$ sufficiently large.
Hence, by Lemma~\ref{lem:Chernoff} and Fact~\ref{factBinom},
\[ \pr(\deg_G(v) \leq \tfrac14 \pi r^2 + \tfrac12 \omega r \mid \X_v \in C_n) \, \leq  \, e^{-\Theta(j_1^2)}.\]
Since $\omega \ll r/\sqrt{\log n}$, we have $j_1^2 \geq (r/\omega)^2 \gg \log n$. 
Thus $n e^{-\Theta(j_1^2)} = o(1)$, and so 
 whp there is no vertex $v$ such that $\X_v \in C_n$ and $\deg_G(v) \leq \tfrac14 \pi r^2 + \tfrac12 \omega r$.

We need a little more care near the corners.
Let $i \in [4]$ and let $j$ be an integer with 
$ j_0 \leq j \leq j_1$.  The area of 
$B_i^{j+1} \setminus B_i^j$ is $\frac14 \pi (2j+1)$.  For each point 
$\x \in B_i^{j+1} \setminus B_i^j$,  $\x$ is at distance at least $j$ from each corner of $\SR$, so by Lemma~\ref{lem.extraarea} (ii)
we have 
\[ \sigma_n(\x) \geq \sigma^{(j)} := \tfrac14 \pi r^2 + j(r-j/2) = \tfrac14\pi r^2 + (1+o(1)) jr.\]
Also, 
\[ \tfrac14 \pi r^2 + \tfrac12 \omega r \leq \tfrac{n-1}{n} \sigma^{(j)}  - (\tfrac12 +o(1)) jr.\]
Thus, by Lemma~\ref{lem:Chernoff} and Fact~\ref{factBinom},
\[ \pr\big(\deg_G(v) \leq \tfrac14 \pi r^2 + \tfrac12 \omega r \mid \X_v \in B_{i}^{j+1} \setminus B_{i}^j \big) \le e^{-\Theta(j^2)}.\] 
Therefore, for each $i \in [4]$,
\begin{eqnarray*}
&&
\pr\big(\mbox{for some } v \in V,\, (\X_v \in B_{i}^{j_1} \setminus B_{i}^{j_0}) \land ( \deg_G(v) \leq  \tfrac14 \pi r^2 + \tfrac12 \omega r) \big)\\
& \leq &
\sum_{v \in V} \sum_{j=j_0}^{j_1-1} \pr(\X_v \in B_{i}^{j+1} \setminus B_{i}^j) \, \pr\big(\deg_G(v) \leq  \tfrac14 \pi r^2 + \tfrac12 \omega r \mid \X_v \in B_{i}^{j+1} \setminus B_{i}^j \big)\\
& \leq &
n \, \sum_{j \geq j_0} \frac{\tfrac14 \pi (2j+1)}{n} e^{-\Theta(j^2)} \; = \; o(1).
\end{eqnarray*}
Hence whp $\deg_G(v)> \tfrac14 \pi r^2 + \tfrac12 \omega r$ for each vertex $v$ such that $\X_v$ is not in one of the four corner regions $B_i^{j_0}$;
 and so we have completed the proof of part~(b).
\end{proof}

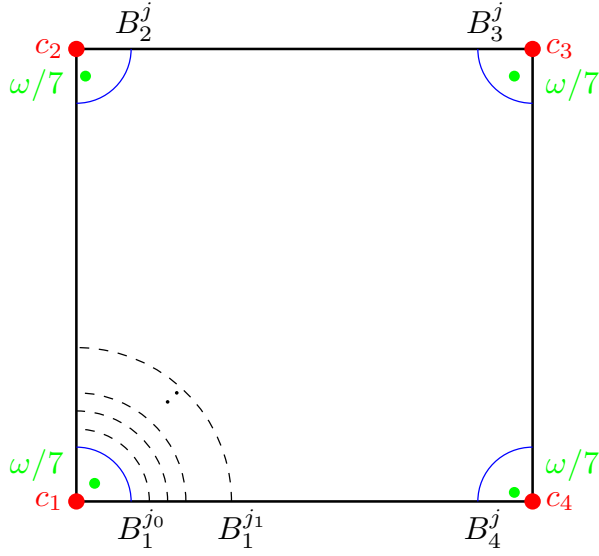
\begin{figure}[h] 
\scalebox{1.2}{%
\begin{tikzpicture}
\coordinate (A) at (0,0);
\coordinate (B) at (5,0);
 \coordinate (C) at (5,5);
 \coordinate (D) at (0,5);
\coordinate (H1) at (0.6, 0);
\coordinate (H2) at (0, 0.6);
\coordinate (I1) at (4.4, 0);
\coordinate (I2) at (5, 0.6);
\coordinate (J) at (4.4,4.4);
\coordinate (J1) at (5,4.4);
\coordinate (J2) at (4.4,5);
\coordinate (K1) at (0,4.4);
\coordinate (K2) at (0.6, 5);
\coordinate (H1) at (0.6, 0);
\coordinate (H2) at (0, 0.6);
\coordinate (L1a) at (0.8, 0);
\coordinate (L1b) at (0.8, 0.8);
\coordinate (L1c) at (0, 0.8);
\coordinate (L2a) at (1, 0);
\coordinate (L2b) at (1,1);
\coordinate (L2c) at (0, 1);
\coordinate (L3a) at (1.2, 0);
\coordinate (L3b) at (1.2,1.2);
\coordinate (L3c) at (0, 1.2);
\coordinate (L4a) at (1.7, 0);
\coordinate (L4b) at (1.7,1.7);
\coordinate (L4c) at (0, 1.7);
\coordinate (Pa) at (1,1.1);
\coordinate (Pb) at (1.1,1.2);
\coordinate (Pc) at (1.5,1.5);
\coordinate (Pd) at (1.6,1.6);
 \coordinate (Aa) at (0.2,0.2);
\coordinate (Bb) at (4.8,0.1);
\coordinate (Cc) at (4.8,4.7);
\coordinate (Dd) at  (0.1,4.7);
\draw [thick] (A)--(B)--(C)--(D)--(A);
\fill[green]  (Aa) circle [radius=0.06];
\fill[green] (Bb) circle [radius=0.06];
\fill[green] (Cc) circle [radius=0.06];
\fill[green] (Dd) circle [radius=0.06];
\fill[red]  (A) circle [radius=0.09];
\fill[red] (B) circle [radius=0.09];
\fill[red] (C) circle [radius=0.09];
\fill[red] (D) circle [radius=0.09];
\draw [blue] (H1)  to [out=90,in=0]  (H2);
\draw [blue] (I1)   to [out=90,in=180]  (I2);
\draw [blue] (J1) to [out=180,in=270]  (J2);
 \draw [blue] (K1) to [out=0,in=270]  (K2);
\fill[black] (Pa) circle [radius=0.02];
\fill[black] (Pb) circle [radius=0.02];
\draw [black,dashed] (L1a)  to [out=90,in=0] (L1c);
\draw [black,dashed] (L2a) to [out=90,in=0]  (L2c);
\draw [black,dashed] (L3a) to [out=90,in=0] (L3c);
\draw [black,dashed] (L4a) to [out=90,in=0] (L4c);
\node[right] at (0.3,-0.3) {\footnotesize{$B_1^{j_0}$}};
\node[right] at (4.1,-0.3) {\footnotesize{$B_4^j$}};
\node [left] at (1,5.3) {\footnotesize{$B_2^j$}};
\node [left] at (4.85,5.3) {\footnotesize{$B_3^j$}};
\node [right] at (1.4,-0.3) {\footnotesize{$B_1^{j_1}$}};
\node[left, green] at (0,0.4) {\footnotesize{$\omega/7$}};
\node[right, green] at (5,0.4) {\footnotesize{$\omega/7$}};
\node[right, green] at (5,4.6) {\footnotesize{$\omega/7$}};
\node[left, green] at (0,4.6) {\footnotesize{$\omega/7$}};
\node[right,red] at (B) {\footnotesize{$c_4$}};
\node[left,red] at (A) {\footnotesize{$c_1$}};
\node[right,red] at (C) {\footnotesize{$c_3$}};
\node[left,red] at (D) {\footnotesize{$c_2$}};
\end{tikzpicture}}
\caption{Choosing points in the 4 corners of the $\sqrt{n}\times \sqrt{n}$ square $\SR$}\label{fig:corner}
\end{figure}

The above lemma shows us how to find $4$ vertices $v$ 
such that whp the corresponding points $\X_v$ are close to the four corner points $c_i$ of $\SR$.
\begin{lemma}\label{lem:corner}
Let $r=r(n)$ satisfy
$\sqrt{\log n} \ll r \ll \sqrt{n}$,
 and consider the random geometric graph $G=\RG$. %
Let $\omega=\omega(n)$ be any function tending to infinity as $n \to \infty$. 
There is a polynomial-time (in $n$) algorithm which, on input $A(G)$, finds four vertices $v_1,v_2,v_3,v_4$ such that whp the following holds: for some (unknown) symmetry $\pi$ of $\SR$, 
\[ \X_{v_i} \in B(\pi(c_i), \omega) \;\; \mbox{ for each } i \in [4]. \] 
\end{lemma}

\begin{proof}
Consider the following algorithm: pick a vertex of minimal degree, call it $u_1$, and mark $u_1$ and all its neighbours. Continue iteratively on the set of unmarked vertices, until we have found four vertices $u_1,\ldots, u_4$.
(Whp each vertex has degree at most $1.1\, \pi r^2$; so after at most 3 steps, at most $3.3 \,\pi r^2 +3=o(n)$ vertices are marked, and so whp we will find $u_1,\ldots,u_4$.)
Let $u_1'$ be a vertex amongst $u_2,u_3,u_4$ maximising the graph distance from $u_1$, and list the four vertices as $v_1,v_2,v_3,v_4$ where $v_1=u_1$ and $v_3=u_1'$ (and $v_2$ and $v_4$ are the other two of the vertices $u_i$ listed in some order).  We shall see that whp $v_1,v_2,v_3,v_4$ are as required.

By Lemma~\ref{lem.corners}, whp the vertices $u_1,\ldots,u_4$
are each within distance $\omega$ of a corner of $\SR$, and the marking procedure ensures that the four corners involved are distinct.
If $u_i$ and $u_j$ 
are such that $\X_{u_i}$ and $\X_{u_j}$ are within distance $\omega$ of opposite corners of $\SR$, then $\dE(\X_{u_i},\X_{u_j}) \geq \sqrt{2n}-2\omega$ and so $d_G(u_i, u_j) \geq (1 \, + \, o(1)) \sqrt{2n}/r$. 
If $\X_{u_i}$ and $\X_{u_j}$ are within distance $\omega$ of adjacent corners, then $\dE(\X_{u_i},\X_{u_j}) \leq \sqrt{n}+ \omega$; and so, since we may assume wlog that $\omega \ll \sqrt{n}$, whp $d_G(u_i, u_j) \leq (1+o(1)) \sqrt{n}/r$ by 
Corollary~\ref{cor.dists}. 
Hence, whp $u_1 = v_1$ and $u'_1 = v_3$ are within distance $\omega$ of opposite corners, as are the other two of the chosen vertices.  For each $i \in [4]$, denote the corner closest to $\X_{v_i}$ by $c_{\sigma(i)}$.  Then
whp $\sigma$ is a permutation of $[4]$, and $c_{\sigma(1)}$ and $c_{\sigma(3)}$ are opposite corners; and so $c_{\sigma(1)},\ldots, c_{\sigma(4)}$ lists the corners of $\SR$ in either clockwise or anticlockwise order. Thus $\sigma$ extends to a (unique) symmetry $\pi$ of $\SR$, and we are done.
 \end{proof}
\vspace{0.3cm}

Having found four vertices $v_1,\ldots, v_4$ such that the points
$\X_{v_i}$ are close to the four corner vertices of $\SR$, for each other vertex $v \in V(G)$ we will be able to use the graph distances {from $v$ to each of} $v_1,\ldots,v_4$ to obtain an approximation to $\X_v$.

\section{Estimating $r$ and Euclidean distances}
\label{sec.randdErev}
In this section, we use the preliminary results from the last section to see how to estimate the threshold distance $r$, and Euclidean distances between points, sufficiently accurately to be able to prove Theorem~\ref{thm:main} in the next section.
We start by considering the number of edges in $G\in\RG$ (with more precision than will be needed here).

\subsection{Expectation and variance of the number of edges in $G\in\RG$}

\begin{proposition} \label{prop.M}
Let $n$ be a positive integer, let $0 < r \leq \sqrt{n}$, and let $M$ be the random number of edges in $G\in\RG$.  Then
\begin{equation} \label{eqn.EM}
  \E[M] = \tfrac12 (n\!-\!1) \pi r^2 \big(1- \tfrac8{3\pi} (r/\sqrt{n}) + \tfrac{1}{2 \pi} (r^2/n) \big)
\end{equation}
and
\begin{equation} \label{eqn.varM}
  \var(M) < \tfrac12\pi nr^2 + 4 \pi^2 \sqrt{n} r^5.
\end{equation}
\end{proposition}
\noindent
Observe that, by~(\ref{eqn.EM}), for $0<r=r(n) \ll \sqrt{n}$ 
in particular it follows that
\begin{equation} \label{eqn.EM2}
  \E[M] = (1+o(1)) \tfrac12 \pi n r^2.
\end{equation}

We will use one lemma in the proof of this proposition.
For each point $z$ in the square $[0,1]^2$, let $f(z)$ be the area of the ball $B(z,1)$ to the right of the line $x=1$, and let $g(z)$ be the area of $B(z,1) \cap [1,2]^2$.
\begin{lemma} \label{lem.outside}
Let the random variable $Z$ be 
uniformly distributed over the square $[0,1]^2$.  Then 
\begin{itemize}
\item $\E[f(Z)] = \alpha = \tfrac23$ 
\item  $\E[g(Z)] = \beta = \tfrac18$. 
\item $\E[B(Z,1) \cap [0,1]^2] =\pi  - 4 \alpha + 4 \beta \approx 0.974926.$
\end{itemize}
\end{lemma}
\begin{proof}

Let $X_1$ be uniformly distributed on $[0,1]$.  Then $\E[f(Z)]$ is the expected area of $B((1-X_1,0),1)$ to the right of the line $x=1$, which is
\begin{eqnarray*}
 \int_{0}^{1} \left(\int_{x}^{1} 2\sqrt{1-t^2}\, dt\right) dx
&=&
\int_{0}^{1} \left(\int_{0}^{t} 2\sqrt{1-t^2}\, dx\right) dt\\
&=&
\int_{0}^{1}  2t \sqrt{1-t^2}\, dt\\
&=&
[- \tfrac23 (1-t^2)^{3/2}]_{0}^{1}\;\; = \;\; \tfrac23.
\end{eqnarray*}

Now let $X_1,X_2,Y_1,Y_2$ be independent random variables, each uniformly distributed on $[0,1]$.
Let $X = X_1+X_2$ and $Y=Y_1+Y_2$.
Then $\pr(X \leq t) = \tfrac12 t^2$ for $0 \leq t \leq 1$, 
and therefore we may take the density function $f_0$ of $X$ to satisfy $f_0(t) = t$ for $0 \leq t \leq 1$. 
Observe that $Z_1=(1-X_1,1-Y_1)$ is uniformly distributed over $[0,1]^2$, $Z_2=(1+X_2,1+Y_2)$ is uniformly distributed over $[1,2]^2$, and $d_E(Z_1,Z_2)^2 = X^2 + Y^2$.
Then
\begin{eqnarray*}
\E[g(Z)] &=&
\pr(X^2 + Y^2 \leq 1)
\\
&=&
\int_{0}^{1} \pr(Y \leq \sqrt{1-x^2}) f_0(x)\, dx\\
&=&
\int_{0}^{1} \tfrac12 (1-x^2) \, x\, dx \\
&=&
\tfrac12 [x^2/2 - x^4/4]_{0}^{1} \;\; = \;\; \tfrac18,
\end{eqnarray*}
as required. 
\end{proof}

\begin{proof}[Proof of Proposition~\ref{prop.M}]
For $u \neq v$ in $[n]$, let $Y(u,v)$ be 1 if $u$ and $v$ are adjacent in $G$, and be 0 otherwise.  Let $p = \E[Y(u,v)]$. Then $M =  \sum_{u<v} Y(u,v)$, and $\E[M] = \binom{n}{2} p$.  

Now let $Z$ be uniformly distributed over the square ${\mathcal S}_n = [-\frac12 \sqrt{n}, \frac12 \sqrt{n}]^2$.
Observe that $Z$ is at distance at most $r$ from a given side of the square ${\mathcal S}_n$ with probability $r/\sqrt{n}$, and $Z$ is in the $r \times r$ subsquare touching a given corner point of ${\mathcal S}_n$ with probability $r^2/n$. Thus, by Lemma~\ref{lem.outside},
the expected area of $B(Z,r) \cap {\mathcal S}_n$ is
\[ \pi r^2 - 4 (r/\sqrt{n})\, \alpha r^2 + 4 (r^2/n) \beta r^2 = \pi r^2 - \tfrac83 (r/\sqrt{n}) r^2 + \tfrac12 (r^2/n) r^2.\]
But this equals $np$, so
\[ \E[M] = \tfrac12 (n\!-\!1) np = \tfrac12 (n\!-\!1) \pi r^2 \big(1- \tfrac8{3\pi} (r/\sqrt{n}) + \tfrac{1}{2 \pi} (r^2/n) \big),\]
 as required in~(\ref{eqn.EM}).

Now consider the variance of $M$.
We may expand $\E[M^2]$ as
\begin{equation} \label{eqn.expand}
\E[M^2] = \sum_{\{u,v\}} \sum_{\{u',v'\}} \E[Y(u,v) Y(u',v')]
\end{equation}
where $\{u,v\}$ and $\{u',v'\}$ are both unordered pairs of distinct vertices.  If the pairs $\{u,v\}$ and $\{u',v'\}$ are identical, then $\E[Y(u,v) Y(u',v')]=\E[Y(u,v)]=p$. If the vertices $u,v,u',v' \in [n]$ are all distinct then $\E[Y(u,v) Y(u',v')]=p^2$. Now suppose that $u,v,w \in [n]$ are distinct.  Then
\[ \E[Y(u,v) Y(u,w)] = \E[Y(u,v)]\, \E[Y(u,w) | Y(u,v)=1] \leq p \cdot \pi r^2/n.\]
Call a point in the square ${\mathcal S}_n$ \emph{central} if it is not within distance $r$ of a side of the square.  Then 
\[ p > \E[Y(u,v) | \X_u \mbox{ central}] \pr{[\X_u \mbox{ central}]} >  (\pi r^2/n) (1-4r/\sqrt{n}) = \pi r^2/n - 4 \pi (r/\sqrt{n})^3. \]
Thus
\[ \E[Y(u,v) Y(u,w)] <  p \, ( p + 4 \pi  (r/\sqrt{n})^3) \leq p^2 + 4 \pi^2 (r/\sqrt{n})^5 \]
(using $p \leq \pi r^2/n$).  But there are at most $n^3$ such terms in the sum in~(\ref{eqn.expand}), so
\[ \E[M^2]  <  \E[M] + (\E[M])^2 + n^3 \cdot 4 \pi^2 (r/\sqrt{n})^5. \]
Hence, since $\E[M] < \frac12 \pi nr^2$,
\[ \var(M) < \tfrac12\pi nr^2 + 4 \pi^2 \sqrt{n} r^5, \]
as required in~(\ref{eqn.varM}).
\end{proof}

If $r \gg n^{1/6}$ then both $n^{1/2}r$ and $n^{1/4}r^{5/2}$ are $\ll (r^2/n) nr^2 = r^4$, and so by Chebyshev's inequality
\[ M=\tfrac12 n \pi r^2 \big(1- \tfrac8{3\pi} (r/\sqrt{n}) + \tfrac{1+o(1)}{2 \pi} (r^2/n) \big) \mbox{ whp}.\]
If $nr^2 \to \infty$ and $r \ll \sqrt{n}$ then similarly
\[ M=\tfrac{1+o(1)}2 n \pi r^2 \mbox{ whp}.\]

\subsection{The estimator $\hat{r}\, $: proof of Proposition~\ref{prop.rhat}}


Let $\mu(n,r)$ denote $\E[M]$, given in equation~(\ref{eqn.EM}), and note that $\mu(n,r)$ is increasing in $r$. 
We can determine $M$ from the adjacency matrix in $O(n^2)$ time; and then, by repeated bisection (for example) we can find $\hat{r}$ such that $| \mu(n, \hat{r}) - M| < 1$, in $O(\log n)$ arithmetic operations.  Thus we can calculate the estimator $\hat{r}$ in $O(n^2)$ time.

Let $\sigma=\sigma(n,r)$ denote the standard deviation of $M$, and let $\beta=\beta(n,r) = \sqrt{n} r + n^{1/4} r^{5/2}$.
By~(\ref{eqn.varM}) we have
$ \sigma < \sqrt{\pi/2} \sqrt{n} r + 2 \pi n^{1/4} r^{5/2} < 7 \beta$.
Let $\omega=\omega(n) \to \infty$ as $n \to \infty$, and satisfy $\omega^2 \ll \min \{r,1\} \sqrt{n}$. Let $\omega'=\omega'(n) \to \infty$ as $n \to \infty$, and satisfy $\omega' \ll \omega$.
By Chebyshev's inequality, 
whp $|M-\mu(n,r)| \ll \omega' \sigma$,
and so whp 
$|\mu(n,r) - \mu(n,\hat{r})| < \omega' \sigma +1$.  It follows that
\begin{equation} \label{eqn.mudiff}
|\mu(n,\hat{r}) - \mu(n,r)| \ll \omega \, \beta \;\; \mbox{ whp}.
\end{equation}


We want to bound $x:= | \hat{r}-r|$.  
Let $f(y) = |\mu(n,r+y) - \mu(n,r)|$, and observe that if $y'>y \geq 0$ or $y'<y \leq 0$ then $f(y')>f(y)$. By~(\ref{eqn.EM}), for $y$ such that $|y| = o(r)$ we have  $\mu(n,r+y) = \mu(n,r)(1+ 2y/r +o(|y|/r))$.
Thus, for such a $y$, by~(\ref{eqn.EM2})
\begin{equation} \label{eqn.fy}
f(y) = (2+o(1)) \mu(n,r) |y|/r = (1+o(1)) \pi r n |y| .
\end{equation}

We make two observations. 
(a) For $r \le n^{1/6}$, we have 
$\sqrt{n} r \geq n^{1/4} r^{5/2}$ so $\beta \leq 2 \sqrt{n} r$; and $r^2 n /\omega \gg \omega \sqrt{n} r$ (since $\omega^2 \ll \sqrt{n} r$)), so $r^2 n /\omega \gg \omega \beta$.
(b) For $r \ge n^{1/6}$, we have 
$\sqrt{n} r \leq n^{1/4} r^{5/2}$ so $\beta \leq 2 n^{1/4} r^{5/2}$; and $r^2 n /\omega \gg \omega n^{1/4} r^{5/2}$ (since $\omega^2 \ll \sqrt{n} \ll n^{3/4} r^{-1/2}$), so $r^2 n /\omega \gg \omega \beta$.
This in both cases $r^2 n /\omega \gg \omega \beta$.
Hence, by~(\ref{eqn.fy}), setting $y = \pm r/\omega$,
 we have 
\[ f(y) = (1+o(1)) \pi r^2 n/ \omega \gg
\omega \beta.\]
But now, by~(\ref{eqn.mudiff}) and the observed monotonicity of $f$, whp $|x| < |y|  \ll r$.
It follows that whp $f(x) = (1+o(1)) \pi r n |x| $, so $|x|= (1+o(1)) f(x)/(\pi r n)$; and thus by~(\ref{eqn.mudiff}), whp
\[ |x| \ll \omega  \beta /(\pi r n) \leq \omega \cdot (n^{-1/2} + n^{-3/4} r^{3/2}). \]
Hence
\[ | \hat{r}-r| = |x| \ll \omega \cdot (n^{-1/2} + \rho^{-3/2}) \;\; \mbox{ whp}, \]
which completes the proof of Proposition~\ref{prop.rhat}.

\subsection{The estimator $\hat{r}$ and Euclidean distance}

In this subsection we restrict $r$ to be large enough so that we can use Corollary~\ref{cor.dists2}.

\begin{lemma} \label{lem.dE2}
Let $r=r(n)$ satisfy $\, n^{3/14} \ll r \ll \sqrt{n}$.    
Then there exists $\delta=\delta(n) \to 0$ as $n \to \infty$
such that whp,  for all pairs $u,v$ of vertices, 
\begin{equation} \label{eqn.dEbounds1}
\hat{r} d_G(u,v) + \delta \hat{r} \geq d_E(\X_u,\X_v) \geq \hat{r} d_G(u,v) - (1+ \delta) \hat{r}.
\end{equation}
\end{lemma}
\noindent
Observe that by~(\ref{eqn.dEbounds1}) and Proposition~\ref{prop.rhat} we can estimate each value $d_E(\X_u,\X_v)$ up to an additive error of
$(\frac12 + o(1)) \hat{r} = (\frac12+o(1))r$ whp.
\begin{proof}
By Corollary~\ref{cor.dists2}, there exists $\delta=\delta(n) \to 0$ sufficiently slowly that whp,
for all pairs vertices $u,v$ we have
\begin{equation} \label{eqn.dEbounds2}
r d_G(u,v) \geq d_E(\X_u,\X_v) \geq r d_G(u,v) - (1+\delta/2)r.
\end{equation}
%
%
Let $\rho=\sqrt{n}/r$, as in Proposition~\ref{prop.rhat}. By Corollary~\ref{cor.dists2}, whp 
$$
\max_{u,v} d_G(u,v) \leq \sqrt{2n}/r+2 \le 2 \rho 
$$
for $n$ sufficiently large (where the maximum is over all pairs $u,v$ of vertices). Hence, by~(\ref{eqn.dEbounds2}), 
\begin{eqnarray} \label{eqn.dEbounds3}
\hat{r} d_G(u,v) + 2 \rho\, |\hat{r} -r|
&\geq & d_E(\X_u,\X_v) \nonumber \\
& \geq &
\hat{r} d_G(u,v) - 2 \rho\, |\hat{r} -r| - (1+ \delta/2) r.
\end{eqnarray}
By Proposition~\ref{prop.rhat}, whp $|\hat{r} -r| \leq \omega \rho^{-3/2}$.
Thus whp
$$
2\rho |\hat{r} -r| =o(1),
$$
by choosing $\omega$ to be growing sufficiently slowly (which we may, since $\rho \to \infty$). 
We may assume that $\delta \to 0$ sufficiently slowly, so that $\delta r \to \infty$ as $n \to \infty$, and in particular, whp $2\rho |\hat{r} -r| =o(\delta r)$, and thus, since by Proposition~\ref{prop.rhat} we have $\hat{r}/r \to 1$ whp, we have 
$2\rho |\hat{r} -r| \le \delta \hat{r}$ whp.
Also, by the same Proposition and again choosing $\delta$ tending to $0$ with $n$ sufficiently slowly, whp
$$
(1+ \delta/2) r = (1+\delta/2) (1+o(1))\hat{r} \le (1+\delta)\hat{r}.
$$
Putting these bounds into~(\ref{eqn.dEbounds3}) completes the proof of the lemma. 
%
%
\end{proof}

\section{Proof of Theorem~\ref{thm:main}}
\label{sec:Main}
In this section, we prove Theorem~\ref{thm:main}, on the reconstruction of random geometric graphs.

Throughout this section, let $\omega=\omega(n)$ be any function tending to infinity slowly as $n \to \infty$, and in particular such that $\omega \ll \sqrt{\log n}$.
We shall assume at various places without further comment that $n$ is sufficiently large.
Let $\mathcal{B}_1$ be the event that we find vertices $v_1,\ldots,v_4$ such that
$d_E(\X_{v_i},\pi(c_i)) < \omega$ for each $i=1,\ldots,4$, for some (unknown) random symmetry $\pi=\pi(\Psi)$ of~$\SR$.
By Lemma~\ref{lem:corner}, $\mathcal{B}_1$ holds whp.  Let $\sigma_0$ denote the identity symmetry. 
If $\mathcal{B}_1$ does not hold then let us set $\pi=\sigma_0$ (the choice of $\pi$ as $\sigma_0$ will not be important).
Now let $\sigma$ be any given symmetry.  Observe that $\mathcal{B}_1$ holds for $\Psi$ if and only if it holds for $\sigma^{-1} \circ \Psi$;
on $\mathcal{B}_1$, $\pi(\Psi)=\sigma$ if and only if $\pi(\sigma^{-1} \circ \Psi)=\sigma_0$, and $\Psi$ and $\sigma^{-1} \circ \Psi$ have the same distribution.  Thus for each symmetry $\sigma$
\begin{equation} \label{eqn.pi}
 \pr(\mathcal{B}_1 \land (\pi =\sigma)) =  \pr(\mathcal{B}_1 \land (\pi =\sigma_0)),
\end{equation}
and so $\pr (\pi = \sigma) \to \tfrac18$ as $n \to \infty$.
Since we are using the symmetry-adjusted measure $d^*$, we may treat the random symmetry $\pi$ as if it were the identity, as we shall check below.
We set $\Phi(v_i)=c_i$, and still have to assign $\Phi(v)$ for all other vertices $v \in V(G)$. 
Let $\mathcal{B}_2$ be the event that $\mathcal{B}_1$ holds and $\pi$ is the identity.  Thus $\pr(\mathcal{B}_2) \to \tfrac18$ as $n \to \infty$.
Recall that we are given $\eps >0$: we may assume wlog that $\eps< \frac12$ say.
The main step in the proof will be to show that $\Phi$ can be defined for all vertices in such a way that
\begin{equation} \label{claimd}
\mbox{conditional on } \mathcal{B}_2, \mbox{ we have } \dmax(\Psi,\Phi) <(1+\eps)r \mbox{ whp}. 
\end{equation}
(Of course $d^*(\Psi,\Phi) \leq \dmax (\Psi,\Phi)$.)

Let us prove the claim~(\ref{claimd}).
By Proposition~\ref{prop.rhat}, whp $|\hat{r}-r|  \ll \omega \rho^{-3/2}$.  Hence, whp, for each pair $u,v$ of distinct vertices
\[ \hat{r}(d_G(u,v)+\eps/4) \le  r(d_G(u,v)+\eps/4)+\omega \rho^{-3/2}(2d_G(u,v)). \]
%
%
But by Corollary~\ref{cor.dists2}, whp, for each pair $u,v$ of vertices 
\[ rd_G(u,v) \leq d_E(\X_u,\X_v)+ 2r \leq \sqrt{2n}+ 2r \leq 2 \sqrt{n}, \]
 so
\[\omega \rho^{-3/2} d_G(u,v) =
\frac{\omega r^{1/2}\, rd_G(u,v)}{n^{3/4}} 
\leq \frac{2 \omega r^{1/2}}{n^{1/4}}
= 2 \omega \rho^{-1/2} = o(r);
\]
and thus
\[ \hat{r}(d_G(u,v)+\eps/4) \le  r(d_G(u,v)+\eps/3).\]

By the same argument we obtain whp, for each pair $u,v$ of vertices 
\begin{eqnarray*}
\hat{r}(d_G(u,v)-(1+ \eps/4))
 & \ge & r(d_G(u,v)-(1+\eps/4 ))-\omega \rho^{-3/2} d_G(u,v)\\
& \ge & r(d_G(u,v)-(1+\eps/3)).
\end{eqnarray*}
By Lemma~\ref{lem.dE2} with $\delta=\eps/5$, whp, for each pair $u,v$ of vertices 
\[ \hat{r} (d_G(u,v) + \eps/5) \geq d_E(\X_u,\X_v) \geq 
\hat{r} (d_G(u,v) - (1+ \eps/5)).\]
But $\omega + \hat{r} \eps/5 \leq \hat{r} \eps/4$ whp.
Hence, conditional on the event $\mathcal{B}_2$, whp, for each $i \in [4]$ and vertex $v \in V^- = V \setminus \{v_1,\ldots,v_4\}$, we have 
\begin{eqnarray*}
\label{eqn.d_Ebounds5}
&& r (d_G(v,v_i) + \eps/3) \geq \hat{r} (d_G(v,v_i) + \eps/4) \ge d_E(\X_v,\X_{v_i})+\omega\\
& \geq &
d_E(\X_v,c_i) \\
& \geq &
d_E(\X_v,\X_{v_i})-\omega \ge \hat{r} (d_G(v,v_i) - (1+ \eps/4)) \geq r (d_G(v,v_i) - (1+ \eps/3)).
\end{eqnarray*}
Let $\mathcal{B}_3$ be the event that these last inequalities hold, so $\mathcal{B}_3$ holds whp.

Condition on the events $\mathcal{B}_2$ and $\mathcal{B}_3$, 
and let $v_1, v_2, v_3, v_4$ be the `corner' vertices found.
We shall show that $\dmax(\Psi,\Phi)<(1+\eps)r$ 
 (deterministically): this will establish~(\ref{claimd}), since then
\[ \pr( \dmax(\Psi,\Phi) \geq (1+\eps)r \mid \mathcal{B}_2) \leq \pr(\overline{\mathcal{B}_3})/\pr(\mathcal{B}_2) = o(1).\]
By symmetry, we may assume for convenience that $v_i=i$ for each $i \in [4]$.

For each $i \in [4]$, let $Q(n,i)$ denote the quarter of $\SR$ containing the corner~$c_i$.
For each vertex $v \in V^-$ there is a `nearest corner' in terms of graph distance to the vertices $v_1,\ldots,v_4$. We need a lemma concerning such a nearest corner.

%
\begin{lemma} \label{lem.wasclaim}
Fix $j \in [4]$.  For each vertex $v \in V^-$ such that $d_G(v,v_j) =\min_{1 \le i \le 4} d_G(v,v_i)$, the corresponding point $\X_v$ lies within distance at most $r$ of the quarter $Q(n,j)$ of~$\SR$.
\end{lemma}

\begin{proof}[Proof of Lemma~\ref{lem.wasclaim}]
Suppose wlog that $j=4$.  Let $v \in V^-$, and suppose for a contradiction that $\X_v$ is not within distance $r$ of $Q(n,4)$.  Assume that $\X_v \in Q(n,1)$ (we shall consider other cases later).
Let us first check that the minimum value of $d_E(\x,c_4) - d_E(\x,c_1)$ over all points $\x$ in $Q(n,1)$ at distance $\geq r$ from $Q(n,4)$ is attained at $\x = \x^*$ where $\x^*=(-r,0)$. 
To see this, let us observe first that the minimum must be attained for some point $(-r, -\frac{\sqrt{n}}{2}+z)$
 with $z \in [0, \frac{\sqrt{n}}{2}]$, as otherwise one could obtain a smaller solution by shifting horizontally to the right until hitting the line $y=-r$. Next, for a given point $\x=(-r, -\frac{\sqrt{n}}{2}+z)$ we have 
$$
d_E(\x,c_4) - d_E(\x,c_1)=\sqrt{z^2+(\tfrac{\sqrt{n}}{2}+r)^2} - \sqrt{z^2+(\tfrac{\sqrt{n}}{2}-r)^2}.
$$
The derivative with respect to $z$ of the previous expression is 
$$
\frac{z}{\sqrt{z^2+(\frac{\sqrt{n}}{2}+r)^2}} - \frac{z}{\sqrt{z^2+(\frac{\sqrt{n}}{2}-r)^2} },
$$
which is clearly negative, since the denominator in the first term is bigger than in the second one. Hence, $d_E(\x,c_4) - d_E(\x,c_1)$ is decreasing in $z$,
and so it is minimised at $\x=\x^*$, as we wished to show.
Hence, all points $\x$ in $Q(n,1)$ at distance $\geq r$ from $Q(n,4)$ satisfy
\begin{eqnarray}\label{distc4c1}
d_E(\x,c_4) - d_E(\x,c_1) & \geq & d_E(\x^*,c_4) - d_E(\x^*,c_1)\nonumber \\ \nonumber
&=&
\sqrt{(\tfrac{\sqrt{n}}{2})^2+(\tfrac{\sqrt{n}}{2}+r)^2} - \sqrt{(\tfrac{\sqrt{n}}{2})^2+(\tfrac{\sqrt{n}}{2}-r)^2} \\
&=& \sqrt{n/2 + \sqrt{n}r(1+o(1))} - \sqrt{n/2 - \sqrt{n}r(1+o(1))} \nonumber \\
&=& \sqrt{\tfrac{n}{2}}\left(1+\tfrac{r}{\sqrt{n}}(1+o(1))\right)-\sqrt{\tfrac{n}{2}}\left(1-\tfrac{r}{\sqrt{n}}(1+o(1))\right)  \nonumber \\
&=&(\sqrt{2}+o(1))r. 
\end{eqnarray}
But, since $\mathcal{B}_3$ holds,
\[ d_G(v,v_1) \leq d_E(\X_v, c_1)/r + 1+ \eps/3\]
and
\[d_G(v,v_4) \geq d_E(\X_v, c_4)/r - \eps/3.\]
Hence, by~(\ref{distc4c1})
\begin{eqnarray*}
  d_G(v,v_4)
& \geq &
  d_E(\X_v, c_4)/r - \eps/3\\
& \geq &  
  d_E(\X_v, c_1)/r + \sqrt{2} +o(1) -\eps/3\\
& \geq &
  d_G(v,v_1) + \sqrt{2} -1 +o(1) -2 \eps/3 \;\; > \;\; d_G(v,v_1)
\end{eqnarray*}
(for $n$ sufficiently large),
a contradiction.  Thus we cannot have $\X_v \in Q(n,1)$.

The case when $\X_v$ is in $Q(n,3)$ is analogous. Finally consider the case when $\X_v$ is in $Q(n,2)$.
Now we shall check that the minimum value of $d_E(\x,c_4) - d_E(\x,c_2)$ over all points $\x$ in $Q(n,2)$ at distance $\geq r$ from $Q(n,4)$ is again  attained at $\x = \x^* =(-r,0)$ (or at $\x=(0,r)$).
To see this, observe much as before that the minimum must be attained at distance exactly $r$ from $(0,0)$, as otherwise one could obtain a smaller solution by shifting $\x$ along the straight line connecting $\x$ with $(0,0)$, until the distance from $(0,0)$ is exactly $r$. Next, for a given point $\x=(- r \cos \theta, r \sin \theta)$
 with $\theta \in [0, \pi/2]$, $d_E(\x,c_4) - d_E(\x,c_2)$ is equal to
$$
\sqrt{(\frac{\sqrt{n}}{2}+r \cos \theta)^2+(\frac{\sqrt{n}}{2}+r \sin \theta)^2} - \sqrt{(\frac{\sqrt{n}}{2}-r \cos \theta)^2+(\frac{\sqrt{n}}{2}-r \sin \theta)^2}.
$$
The derivative with respect to $\theta$ of the above expression is
$$
\frac{r\sqrt{n}(\cos \theta - \sin \theta) }{2\sqrt{(\frac{\sqrt{n}}{2}+r \cos \theta)^2+(\frac{\sqrt{n}}{2}+r \sin \theta)^2}} - \frac{r\sqrt{n}(\sin \theta - \cos \theta)}{ 2\sqrt{(\frac{\sqrt{n}}{2}-r \cos \theta)^2+(\frac{\sqrt{n}}{2}-r \sin \theta)^2}}.
$$
For $\theta \in [0, \pi/4]$, $\cos \theta \ge \sin \theta$, the derivative is positive (or zero), whereas for $\theta \in [\pi/4, \pi/2]$ the derivative is negative (or zero). Hence, the minimum value of $d_E(\x,c_4) - d_E(\x,c_2)$ is attained at $\x=\x^*=(-r,0)$ (or at $\x=(0,r)$),  as we wished to show.

Therefore, by~(\ref{distc4c1}), all points $\x$ in $Q(n,2)$ at distance $\geq r$ from $Q(n,4)$ satisfy
\begin{eqnarray*} 
d_E(\x,c_4) - d_E(\x,c_2) & \geq & d_E(\x^*,c_4) - d_E(\x^*,c_2)\\
&=& (\sqrt{2}+o(1))r. 
\end{eqnarray*}
The remainder of the argument is as before, and so we have completed the proof of the lemma. 
\end{proof}
We now resume the proof of Theorem~\ref{thm:main} (still assuming that the events $\mathcal{B}_2$ and $\mathcal{B}_3$ hold). Let $v \in V^-$ be such that $d_G(v,v_j)=\min_{1 \le i \le 4} d_G(v,v_i)$, as in the last lemma, and assume 
wlog that $j=4$.
Then by the last lemma, $\min_{i \in [3]} d_E(\X_v,c_i) \geq \sqrt{n}/2 - r$.
 Denote by $\alpha$ the angle $c_2 \X_v c_3$ at $\X_v$ between the segments $\X_v c_2$ and $\X_v c_3$, and by $\beta$ the 
angle $c_1 \X_v c_2$ at $\X_v$ between the segments  $\X_v c_1$ and $\X_v c_2$.  We do not observe $\alpha$ or $\beta$ directly,
 but clearly we have $\pi/4 \leq \alpha, \beta \leq \pi/2+o(1)$ (see Figure~\ref{fig2}, left picture), the bounds being attained if $\X_v$ is $c_4$ 
 and if $\X_v$ is near $(0,0)$, respectively.
\begin{figure}[h]
\scalebox{1.27}{
\begin{tikzpicture}
\coordinate (A) at (0,0);
\coordinate (B) at (4,0);
 \coordinate (C) at (4,4);
 \coordinate (D) at (0,4);
\coordinate (E) at (2, 2);
\coordinate (F) at (2, 0);
\coordinate (G) at (4, 2);
 \coordinate (Aa) at (0.1,0.25);
\coordinate (Bb) at (3.8,0.1);
\coordinate (Cc) at (3.7,3.8);
\coordinate (Dd) at  (0.1,3.7);
\coordinate (H) at (3, 1.5);
\fill[green]  (Aa) circle [radius=0.05];
\fill[green] (Bb) circle [radius=0.05];
\fill[green] (Cc) circle [radius=0.05];
\fill[green] (Dd) circle [radius=0.05];
\fill[red]  (A) circle [radius=0.08];
\fill[red] (B) circle [radius=0.08];
\fill[red] (C) circle [radius=0.08];
\fill[red] (D) circle [radius=0.08];
\fill[red] (H) circle [radius=0.1];
\draw [thick] (A)--(B)--(C)--(D)--(A);
\draw [green,dashed] (F)--(E)--(G);
\draw [blue,dashed,thick] (A)--(H)--(D);
\draw [blue,dashed,thick] (H)--(C);
\draw [dashed] (2.58,1.8) to [out=65,in=135] (3.2,2); 
\draw [dashed] (2.2,2) to [out=220,in=130] (2.2,1.1); 
\node [right] at (1.5,1.6) {\footnotesize{$\beta$}};
\node [above] at (2.85,2.1) {\footnotesize{$\alpha$}};
\node[right,red] at (H) {\footnotesize{$\X_v$}};
\node[right,red] at (B) {\footnotesize{$c_4$}};
\node[left,red] at (A) {\footnotesize{$c_1$}};
\node[right,red] at (C) {\footnotesize{$c_3$}};
\node[left,red] at (D) {\footnotesize{$c_2$}};
\end{tikzpicture}}
\hspace{0.5cm}
\scalebox{1.27}{%
\begin{tikzpicture}
\coordinate (A) at (0,0);
\coordinate (B) at (4,0);
 \coordinate (C) at (4,4);
 \coordinate (D) at (0,4);
\coordinate (E) at (2, 2);
\coordinate (F) at (2, 0);
\coordinate (G) at (4, 2);
 \coordinate (Aa) at (0.1,0.25);
\coordinate (Bb) at (3.8,0.1);
\coordinate (Cc) at (3.7,3.8);
\coordinate (Dd) at  (0.1,3.7);
\coordinate (F1) at (3.15,2.4);
\coordinate (F2) at (3.2,2.23);
\coordinate (F3) at (3.8,1.85);
\coordinate (F4) at (3.95,1.95);
\coordinate (F5) at (3.7,1.3);
\coordinate (F6) at (3.7,1);
\coordinate (F7) at (2.95,0.65);
\coordinate (F8) at (2.7,0.62);
\coordinate (F9) at (2.8,0.85);
\coordinate (F10) at (2,0.99);
\coordinate (F11) at (2.2,1.91);
\coordinate (F12) at (2.58,2.32);
\coordinate (F13) at (2,1.20);
\coordinate (F14) at (3.26,0.75);
\coordinate (F15) at (3.95,1.8);
\coordinate (R1) at (2.96,2.17);
\coordinate (R2) at (3.86,1.85);
\coordinate (R4) at (3.07,0.82);
\coordinate (R6) at (2.125,1.148);
\fill[black]  (R1) circle [radius=0.04];
\fill[black]  (R2) circle [radius=0.04];
\fill[black]  (R4) circle [radius=0.04];
\fill[black]  (R6) circle [radius=0.04];
\coordinate (H) at (3, 1.5);
\fill[green]  (Aa) circle [radius=0.05];
\fill[green] (Bb) circle [radius=0.05];
\fill[green] (Cc) circle [radius=0.05];
\fill[green] (Dd) circle [radius=0.05];
\fill[red]  (A) circle [radius=0.08];
\fill[red] (B) circle [radius=0.08];
\fill[red] (C) circle [radius=0.08];
\fill[red] (D) circle [radius=0.08];
\fill[red] (H) circle [radius=0.1];
\draw [thick] (A)--(B)--(C)--(D)--(A); 
\draw [green,dashed] (A)--(F3);
\draw [green,dashed] (C)--(F8);
\draw [green,dashed] (D)--(F6);
\draw [dashed] (F2)--(F5);
\draw [dashed] (F4)--(F7);
\draw [dashed] (F15)--(F12);
\draw [dashed] (F1)--(F10);
\draw [dashed] (F9)--(F11);
\draw [dashed] (F14)--(F13);
\draw [blue] (R1)--(R2);
\draw [blue] (R4)--(R6);
\draw [red] (R1)--(R6);
\draw [red] (R2)--(R4);

\draw [dashed,magenta] (1.86,2.4) to [out=60,in=120] (3.4,2.48);
\draw [dashed,magenta] (1.86,2.4) to [out=225,in=120] (1.86,1);
\node [right] at (2.1,1.2)  {\scriptsize $\alpha$};
\node [left] at (3.86,1.8) {\scriptsize $\alpha$};
\node [below] at (3,2.3) {\scriptsize{$\pi-\alpha$}};
\node [above] at (3.07,0.7) { \scriptsize{$\pi-\alpha$}};
\node [magenta,above] at (2.5,2.8) {\small $\alpha$};
\node [magenta,left] at (1.6,1.8) {\small $\beta$};
\node[right,red] at (B) {\footnotesize{$c_4$}};
\node[left,red] at (A) {\footnotesize{$c_1$}};
\node[right,red] at (C) {\footnotesize{$c_3$}};
\node[left,red] at (D) {\footnotesize{$c_2$}};
\end{tikzpicture}} 
\caption{Illustration of the notation}\label{fig2}
\end{figure}
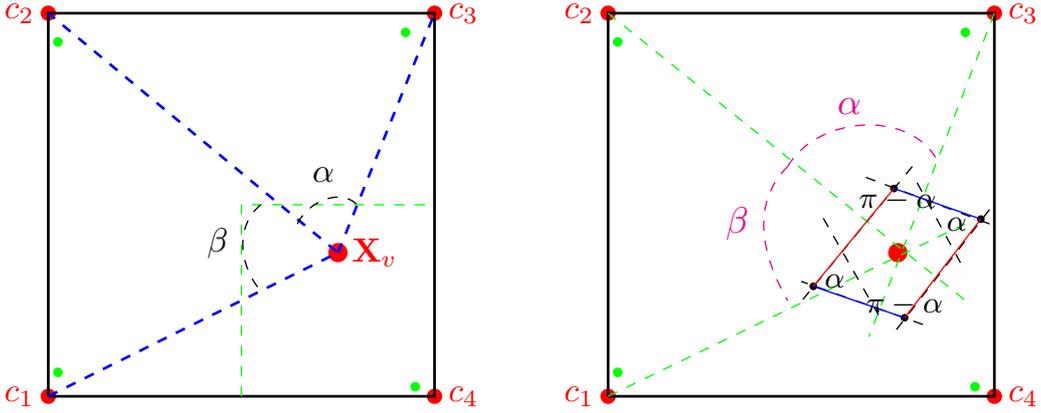

For each $i=1,2,3$, let  
$R_i(v) = \hat{r}(d_G(v,v_i)\!-\!\frac12) \; \big(= \Theta(\sqrt{n})\big)$, and let $C_i(v)$ be the circle centred on the corner $c_i$ with radius $R_i(v)$.
Also, let $A_i(v)$ be the annulus centred on $c_i$ formed by circles of radii $R_i(v) \pm \hat{r}(\frac12 + \frac{\eps}4)$.  
We can construct these three circles. (We can also construct the corresponding annuli -- assuming as we may that we are given a rational $\eps$ -- though we do not need to do so to prove the theorem.)
Note that $\X_v$ must lie in each of the annuli; 
for, since $\mathcal{B}_3$ holds, 
 $\hat{r} (d_G(v, v_i)+\eps/4) \ge d_E(\X_v,c_i) \ge \hat{r}(d_G(v, v_i) - (1 + \eps/4))$, and hence
  $R_i(v) + \hat{r}(\frac12 + \frac{\eps}4) \ge d_E(\X_v,c_i) \ge R_i(v) - \hat{r}(\frac12 + \frac{\eps}4)$.  It is convenient to consider the circles and annuli in pairs.

Consider first the circles $C_2(v), C_3(v)$ and corresponding annuli $A_2(v), A_3(v)$.  The circles intersect below the line $c_2 c_3$ in a point $\Y_{23}(v)$, where the tangents are at angle $\alpha +o(1)$ (and $\pi - \alpha+o(1)$).
The annuli intersect below the line $c_2 c_3$ 
in a set $B_{23}(v)$ 
which is -- up to lower order terms accounting for curvatures -- a parallelogram 
 $RH_{23}(v)$ with (interior) angles $\alpha$ and $\pi-\alpha$  (see Figure~\ref{fig2}, right picture).  (We chose to consider the circles and annuli with centres far from $\X_v$ so that curvatures would be negligible.) In fact, $RH_{23}(v)$ is 
a rhombus,  
 as in each annulus the radii differ by the same value 
$(1+\eps/2) \hat{r}$; and since the heights are equal, the sides must be of equal length. Further, the point $\Y_{23}(v)$ is at the centre of the rhombus (up to lower order terms), where the diagonals cross.
 (It might happen that some part of the rhombus is actually outside $\SR$, but since this makes the region which we know contains $\X_v$ smaller, it is only helpful for us.) 

The circles $C_1(v), C_2(v)$ and corresponding annuli $A_1(v), A_2(v)$ behave in exactly the way described above for $C_2(v), C_3(v)$ and corresponding annuli.  In particular, the annuli $A_1(v)$, $A_2(v)$ intersect to the right of the line $c_1 c_2$ in a set $B_{12}(v)$ which is close to a rhombus $RH_{12}(v)$ with angles $\beta$ and $\pi-\beta$.

Now consider the circles $C_1(v), C_3(v)$ and corresponding annuli $A_1(v), A_3(v)$; and for convenience let us restrict our attention to the case when $\alpha, \beta \leq \pi/3+o(1)$ (so $\X_v$ is not near the centre $(0,0)$ of $\SR$; in the next paragraph we shall see why it suffices to have this assumption on $\alpha$ and $\beta$).
The annuli $A_1(v)$, $A_3(v)$ intersect inside (or near) the bottom right quarter square $Q(n,4)$ in a set $B_{13}(v)$ which is close to a rhombus $RH_{13}(v)$ with angles $\alpha+\beta$ and $\pi-\alpha-\beta$, where both these angles are in the interval between $\pi/2$ and $2 \pi/3 +o(1)$.

Among these three pairs of circles and corresponding rhombi, we will consider one whose angles are closest to $\pi/2$.  Let us check that there must be at least one with angles in the interval $[\pi/3, 2\pi/3]$ -- we call the corresponding rhombus \emph{squarelike}. 
Indeed, suppose that this is not the case for $RH_{12}(v)$ or $RH_{23}(v)$. Then, since $\pi/4 \leq \alpha, \beta \leq \pi/2+o(1)$, we must have $\alpha, \beta < \pi/3$.
Then, however, $\pi/2 \leq \alpha+\beta < 2\pi/3$, and so $RH_{13}(v)$ is the desired squarelike rhombus. 
Further, the maximum distance from the centre $\Y_{13}(v)$ of the rhombus $RH_{13}(v)$ (the intersection of the diagonals) to a point in the set $B_{13}(v)$
 is half the length $d$ of the long diagonal (recall that we assume $n$ sufficiently large, so that we can safely ignore curvature issues and we can approximate $B_{13}(v)$ arbitrarily well by a rhombus). 

Pick a pair of circles and corresponding rhombus such that their angles are closest to $\pi/2$, and
without loss of generality suppose that the rhombus is $RH_{23}(v)$.  We set $\Phi(v)$ to be a point we calculate within distance 1 
of $\Y_{23}(v)$ (or arbitrarily close to $\Y_{23}(v)$).
 Clearly, the further away the angles $\alpha$ and $\pi-\alpha$ are from $\pi/2$, the longer the long diagonal, and we may thus assume the worst case of $\alpha=\pi/3$ and $\pi-\alpha =2\pi/3$. The shorter diagonal of such a rhombus splits it into two equilateral triangles, with height
$(1+\eps/2) \hat{r}$; and thus half the length $d$ of the longer diagonal is also $(1+\eps/2) \hat{r}$, see Figure~\ref{figCalc}.
 Thus in general  
\[ d/2 \leq (1+\eps/2) \hat{r} \leq (1+\eps) r . \] 
 \begin{figure}
\scalebox{1.4}{%
\begin{tikzpicture}
 \coordinate (A) at (0,0); 
\coordinate (B) at (1,0);
 \coordinate (C) at (3.16,0); 
 \coordinate (D) at (4.16,3); 
\coordinate (E) at (1, 3); 
\coordinate (F) at (2.08,1.5); 
\coordinate (G) at (3.6,2.8);
\coordinate (H) at (4, 1.2);
\fill[red] (F) circle [radius=0.04];

\draw [black] (A)--(C)--(D)--(E)--(A);
\draw [blue,dashed] (E)--(B);
\draw [magenta] (A)--(D);
\draw [magenta] (E)--(C);
\draw [black,dashed] (0.5,0)  to [out=90,in=280] (0.3,0.3);
\draw [black,dashed] (2.8,0)  to [out=90,in=180] (3.27,0.4);
\draw [magenta,dashed] (1.9,1.8)  to [out=45,in=135] (2.5,1.8);
\draw [->,blue, dashed] (0.4,1.5) -- (0.97,1.5);
\draw [->,red, dashed] (0.4,1.45) -- (0.7,0.54);
\node [left] at (.5,1.5) {\scriptsize{$(1+\eps/2)\hat{r}$}};
\node [right] at (3,2.1) {$d$};
\node [right] at (0.5,0.2) {\scriptsize{$\pi/6$}};
\node [right] at (2.4,0.5) {\scriptsize{$2\pi/3$}};
\node [right] at (1.85,2.11) {\scriptsize{$\pi/2$}};
\node [right,red] at (F) {\tiny{$\Phi(v)$}};
\end{tikzpicture}}
\caption{Angles in $R_{23}$, for the extreme case $\alpha=\pi/3$}\label{figCalc}
\end{figure}
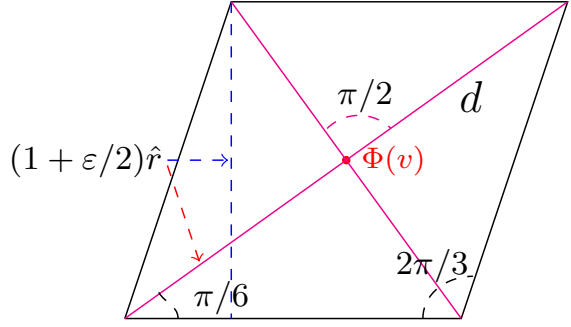

Hence, 
the Euclidean distance from $\Phi(v)$ to any point inside $B_{23}(v)$ is at most $d/2 \leq (1+\eps)r$.
But $\X_v \in B_{23}(v)$, so $d_E(\Phi(v), \X_v) \leq (1+\eps) r$.
This holds for each $v \in V^-$, so
we have found an embedding $\Phi$ with displacement at most $(1+\eps)r$.  (If the point of the intersection of the two diagonals falls outside $\SR$, then we project this point to the closest point on the boundary of $\SR$, and clearly the distance to $\X_v$ can only decrease). 
\smallskip

We have now established~(\ref{claimd}), and it remains only to justify treating the random symmetry $\pi$ as the identity.  We want to replace the conditioning on $\mathcal{B}_2$ in~(\ref{claimd}) by conditioning on $\mathcal{B}_1$.

Let $t>0$ and let $\sigma$ be a symmetry. 
Arguing as for~(\ref{eqn.pi}), and noting also that
$d^*(\Psi,\Phi) = d^*(\sigma^{-1} \circ \Psi,\Phi)$,  we have
\begin{eqnarray*}
&&
\pr\left(\mathcal{B}_1 \land (d^*(\Psi,\Phi) \leq t) \land (\pi(\Psi)=\sigma) \right)\\
&=&
\pr\left(\mathcal{B}_1 \land (d^*(\sigma^{-1} \circ \Psi,\Phi) \leq t) \land (\pi(\sigma^{-1} \circ \Psi)=\sigma_0) \right)\\
&=&
\pr\left(\mathcal{B}_1 \land (d^*(\Psi,\Phi) \leq t) \land (\pi(\Psi)=\sigma_0) \right);
\end{eqnarray*}
so, summing over $\sigma$ we have
\[ \pr\left(\mathcal{B}_1 \land (d^*(\Psi,\Phi) \leq t) \right) = 8 \, \pr\left(\mathcal{B}_2 \land (d^*(\Psi,\Phi) \leq t) \right).\]
But, similarly to~(\ref{eqn.pi}), we have $\pr(\mathcal{B}_1) = 8 \, \pr(\mathcal{B}_2)$, so 
\[ \pr(d^*(\Psi,\Phi) \leq t \mid \mathcal{B}_1) = \pr(d^*(\Psi,\Phi) \leq t \mid \mathcal{B}_2).\]
Hence
\[ \pr(d^*(\Psi,\Phi) \leq (1+\eps)r  \mid \mathcal{B}_1) = \pr(d^*(\Psi,\Phi) \leq (1+\eps)r \mid \mathcal{B}_2) = 1-o(1)\]
by~(\ref{claimd}).  Since $\mathcal{B}_1$ holds whp, this completes the proof of Theorem~\ref{thm:main}. 
\smallskip

\emph{Running time analysis}

The $O(n^2)$ bound on the running time follows from two observations. On the one hand, by Proposition~\ref{prop.rhat}
 at most $O(n^2)$ steps are needed for calculating~$\hat{r}$. On the other hand, the four vertices close to the four corners can be found in time $O(n)$;
and once they have been found, all graph distances from them
can be computed in time $O(n^2)$ (for example, by breadth-first search). If $\eps$ is not rational, then the algorithm can be applied with some rational $\eps'$ such that $0 < \eps' < \eps$.

\section{Proof of Theorem~\ref{thm:main2}}
\label{sec:Main2new}

In this section, we prove Theorem~\ref{thm:main2}, on approximate reconstruction from the 
family of vertex orderings.
As in the algorithm in Theorem~\ref{thm:main}, the algorithm here has two main steps. In the first subsection we give a sketch of the method, in the second subsection we fill in details of step (a), in the third subsection we give several preliminary results needed to analyse step (b), and in the final subsection we complete the proof of the theorem.

\subsection{Sketch of the algorithm}
\label{subsec.sketch}

The algorithm has two main steps. 
\begin{itemize}
\item{Step (a)}
We identify four vertices $v_i$ such that the corresponding points are whp near the four corners of $\SR$, and set $\Phi(v_i)=c_i$. 
\item{Step (b)}
For each other vertex~$v$ we construct two circles, and two corresponding thin annuli both whp containing $\X_v$, centred on a chosen pair of corners, such that the circles 
meet at an angle between $\pi/3$ and $2\pi/3$; and we set $\Phi(v)$ to be the relevant point of intersection of the circles (which is essentially the centre of the rhombus formed by the intersection of the annuli, as before). 
\end{itemize}

We obtain a much smaller displacement error than with random geometric graphs in Theorem~\ref{thm:main} since our annuli are much thinner.  We start with a sketch of the two steps (a) and (b) and of the proofs, before giving the full proofs.  First, however, we introduce some useful notation.
\smallskip

\noindent
\emph{Notation}
For each pair $u, v$ of vertices, we let $k(u,v)$ be the rank of $v$ in the vertex-ordering $\tau_u$.  Thus $k(u,u)=1$; and if $v$ is last in the order $\tau_u$ (farthest from~$u$) then $k(u,v)=n$, and we let $\far(u)=v$.  We assume that, given vertices $u$ and $v$, in constant time we can find the successor of $v$ in the list specifying~$\tau_u$.  Note that by reading through this list we can determine $\far(u)$ in $O(n)$ steps.

For $0 \leq s \leq \sqrt{2}$, let $\lambda(s)$ be the area of the set of points in the unit square $\SU$, centred at $(0,0)$, within distance $s$ of a fixed corner point, say $(-\frac12,-\frac12)$.
Observe that $\lambda$ gives an increasing bijection from $[0,\sqrt{2}]$ to $[0,1]$.  For $0 \leq t \leq 1$ let $s(t) = \lambda^{-1}(t)$, the unique value $s$ in $[0,\sqrt{2}]$ such that $\lambda(s)=t$.  
Also, define $\lambda_n(s) = \lambda(s/\sqrt{n})\, n$ for $0 \leq s \leq \sqrt{2n}$, so
$\lambda_n(s)$ is the area of the set of points in the square $\SR$ within distance $s$ from a fixed corner point; and define $s_n(k) = s(k/n) \sqrt{n}$, and note that $\lambda_n(s_n(k))=k$.

It will be convenient here to say that a sequence $A_n$ of events holds \emph{with very high probability} (wvhp) if $\pr(A_n) = 1-o(1/n)$ as $n \to \infty$.
Finally, let $\omega=\omega(n) \to \infty$ slowly, and in particular assume as in the previous section that $\omega \ll \sqrt{\log n}$.
\bigskip

\noindent
\emph{Sketch of step (a): finding points near the corners of $\SR$}

We call a pair of vertices $\{v, v'\}$ \emph{extreme} if $v'=\far(v)$ and $v=\far(v')$.
in the ordering $\tau_{v'}$.  
We show that for each pair of opposite corners $c, c'$ whp there is an extreme pair $\{v, v'\}$ with $\X_v$ close to $c$ and $\X_{v'}$ close to $c'$, and we can find such extreme pairs quickly. Thus the following event $\mathcal{C}_1$ holds whp.


Let $\mathcal{C}_1$ be the event that this procedure yields vertices $v_1,\ldots,v_4$ such that, for some (unknown, random) 
symmetry $\pi$ of $\SR$, we have $\X_{v_i} \in B(\pi(c_i), \omega)$ for each $i \in [4]$ (so $\X_{v_i}$ is very close to the corner $\pi(c_{i})$).
Also, for given distinct vertices $v_1,\ldots,v_4\,$ let $\mathcal{C}_1(v_1,\ldots,v_4)$ be the event that $\mathcal{C}_1$ holds with this choice of the `corner' vertices.\\

\noindent
\emph{Sketch of step (b): constructing the circles and annuli}

Suppose that the event $\mathcal{C}_1(v_1,\ldots,v_4)$ holds.
We use the orders $\tau_{v_1}, \ldots, \tau_{v_4}$ to estimate the distances from the corners.  

Let $V^- = V \setminus \{v_1,\ldots,v_4\}$.
For each vertex $v \in V^-$, we define $i_0=i_0(v)$  to be the least $j \in [4]$ 
such that $k(v_{j},v) = \min_{i \in [4]} k(v_i,v)$ (picking the least $j$ is just a tie-breaker).
(Thus $\pi(c_{i_0})$ is likely to be the closest corner to $\X_v$.)  Fix $v \in V^-$, and let $I^-=[4]\, \backslash \{i_0\}$.  We consider 
 the three orders $\tau_{v_i}$ for $i \in I^-$.
(We do not use $\tau_{v_{i_0}}$, and do not consider distances from $\pi(c_{i_0})$, so that we work only with `large' distances, and thus we do not need to worry about curvature, exactly as before.)
We want to find a pair of thin annuli, centred on two of the three corners near the points $\X_{v_i}$ for $i \in I^-$, such that wvhp
the `near-rhombus' formed by the intersection 
of the two annuli is squarelike, and wvhp $\X_v$ is in this near-rhombus.

We shall see that, for each $i \in I^-$, we have $k(v_i,v) > 0.19 \, n$ 
wvhp (so we will work only with `large' distances); and for $i=i_0 \pm 1$, we have $k(v_i,v) < 0.91 n$ wvhp. 
(Indices in $[4]$ are always taken mod 4.)
Let $\alpha_0 = \frac{\pi}{9}+\frac{1}{\sqrt{3}} \approx 0.9264$: 
later we shall choose a rational constant $\alpha$ slightly bigger than~$\alpha_0$.
When $k(v_i,v)$ is 
at least $0.19n$ 
and at most $\alpha n$, we have a good estimate of $d_E(\pi(c_i),\X_v)$ (see Lemma~\ref{lem.conc}).
Also, as in the proof of Theorem~\ref{thm:main}, it suffices to consider the case when $\pi$ is the identity map. 

There are two cases depending on the rank $k(v_{i_0+2},v)$ (note that $v_{i_0}$ and $v_{i_0+2}$ are at opposite corners of $\SR$, and so $k(v_{i_0+2},v)$ is likely to be the largest of the values $k(v_i,v)$ for $i \in I^-$):
case~(i) when $k(v_{i_0+2},v) \leq \alpha n$, and case~(ii) when $k(v_{i_0+2},v) > \alpha n$.
In case~(i) we form three circles and three thin annuli, and then choose a best pair of them, as in the proof of Theorem~\ref{thm:main}.  In case~(ii), 
we just use the two circles and thin annuli centred on the corners $c_{i_0-1}$ and $c_{i_0+1}$ (see Lemma~\ref{lem.squarelike}).

\subsection{Filling in the details for step (a)}\label{subsec.a}
We need to show that the method sketched above works.
We first consider step (a), and show that indeed the event $\mathcal{C}_1$ holds whp.  We need one deterministic preliminary lemma.
\begin{lemma} \label{lem.farpoints}
Let $\x \in \SR$, and let
\[ t = \max\{d_E(\x,\y): \y \in \SR\} =
\max_i d_E(\x,c_i)\]
(and note that $t \geq \sqrt{n/2}$).
Then (assuming that $n$ is sufficiently large)
\[ \max \{d_E(\x,\y): \y \in (\SR \setminus \cup_i B^o(c_i, \omega))\} \leq t - \omega/3. \]
(Here $B^o$ denotes an open ball.)
\end{lemma}
%
\begin{proof}[Proof of Lemma~\ref{lem.farpoints}]
Suppose wlog that $\x$ is in the bottom left quarter of $\SR$ (containing $c_1$).  It is easy to see that $d_E(\x,c_3)=t$, and
$\max \{d_E(\x,\y): \y \in (\SR \setminus \cup_i B^o(c_i, \omega))\}$ is achieved at some point $\y \in \SR$ with $d_E(\y,c_3)=\omega$. 
 
Let $c_3-\x = (a,b)$, so $\sqrt{n/2} \leq a,b \leq \sqrt{2n}$, and $t= \sqrt{a^2 +b^2}$.  Suppose further wlog that $a \geq b$ (that is, $\x$ lies on or above the line $y=x$), and note that $a \leq 2b$.

Consider a point $\y$ with $d_E(c_3,\y)=\omega$.  We claim that
\begin{equation} \label{eqn.y}
d_E(\x,\y) \leq t-(1+o(1)) \omega/\sqrt{5}.
\end{equation}
To see this, write $c_3-\y=(p,q)$.  Then $p,q \geq 0$ and $p^2+q^2=\omega^2$, so $p+q \geq \omega$; and we have
\begin{eqnarray*}
d_E(\x,\y) &=&
\big( (a-p)^2 + (b-q)^2 \big)^{1/2}\\
&=&
\big( t^2 - (1+o(1))(2ap+2bq) \big)^{1/2}\\
&=&
t \big( 1 - (1+o(1))(ap+bq)/t^2 \big)\\
&=&
t  - (1+o(1))(ap+bq)/t.
\end{eqnarray*}
But $a \geq b$ and $p+q \geq \omega$, so
\[ d_E(\x,\y) \leq t-(1+o(1)) b(p+q)/t \leq t- (1+o(1)) b \omega /t.\]
Also, $a \leq 2b$ so $b/t \geq b / \sqrt{4b^2+b^2} = 1/\sqrt{5}$. The claim~(\ref{eqn.y}) now follows, and this completes the proof of the lemma.
\end{proof}

\noindent
\emph{Finding points near the corners of $\SR$}

Let $\mathcal{C}_2$ be the event that, for each $i \in [4]$, there is a vertex $u_i$ such $\X_{u_i} \in B(c_i, \omega/4)$.  Then 
$\mathcal{C}_2$ holds whp.
To see this, note that, for a fixed $i \in [4]$
\[ \pr(\X_v \not\in B(c_i, \omega/4) \mbox{ for each } v) = (1- \tfrac14 \pi (\omega/4)^2/n)^n < e^{- \tfrac{\pi}{64} \omega^2} = o(1);\]
and use a union bound.  From now on, assume that $\mathcal{C}_2$ holds.
%
%
We claim that, for each vertex $v$, we have
\begin{equation} \label{eqn.nearcorners}
 \X_{\far(v)} \in \bigcup_{i=1}^4 B(c_i, \omega).
\end{equation}
To see this, 
suppose wlog that $c_1$ is the corner farthest from $\X_v$: then
\[ d_E(\X_v,\X_{\far(v)}) \geq d_E(\X_v, c_1) - \omega/4,\] 
since there is a point $\X_u$ in $B(c_1,\omega/4)$; and so, by Lemma~\ref{lem.farpoints}, for each vertex $v'$ such that 
$\X_{v'}\not\in \bigcup_{i=1}^4 B(c_i, \omega)$,
\[ d_E(\X_v,\X_{v'}) \leq d_E(\X_v, c_1) - \omega/3 < d_E(\X_v,\X_{\far(v)}),\]
which establishes the claim~(\ref{eqn.nearcorners}).

Pick any vertex $v$, let $u=\far(v)$, and let $u'=\far(u)$.  Then by~(\ref{eqn.nearcorners}) there is a pair of opposite corners $c, c'$ such that $\X_u \in B(c,\omega)$ and $\X_{u'} \in B(c',\omega)$.  We may take $v_1$ as $u$ and $v_3$ as $u'$.  Let us suppose wlog that $c$ is $c_1$ and $c'$ is~$c_3$.

Now we must find two vertices such that the corresponding points are close to the opposite corners $c_2$ and $c_4$, and this needs more care.  Let $r = 0.9 \sqrt{n}$ (the exact value $0.9$ is not important), and let $B_i = B(c_i, r)$.  
We claim that, if $\x \in \SR \setminus (B_1 \cup B_3)$ then
\begin{equation} \label{eqn.DiffOfMaxs} 
\max\{ d(c_2,\x), d(c_4,\x)\} -  \max \{ d(c_1,\x), d(c_3,\x)\} = \Omega(\sqrt{n}).
\end{equation}
\begin{proof}[Proof of claim (\ref{eqn.DiffOfMaxs})]
The circles of radius $r$ centred on $c_1$ and $c_3$ meet twice, at points $P$ and $Q$ on the diagonal line between $c_2$ and $c_4$, both at distance $t$ from the origin, where $t^2 = r^2-n/2 = 0.31 n$. The first maximum value is minimised at these points $P$ and $Q$. Thus the first maximum value is at least $t + \sqrt{n/2} \approx 1.264 \sqrt{n}$.
The second maximum value is maximised at the four points on the sides of $\SR$ at distance 
 $0.1\sqrt{n}$ from $c_2$ and $c_4$. Thus the second maximum value is at most $\sqrt{1^2 +0.1^2} \sqrt{n} \approx 1.005 \sqrt{n}$.  Hence the difference between the maxima is at least $0.25 \sqrt{n}$, proving~(\ref{eqn.DiffOfMaxs}).
\end{proof}

It follows now from~(\ref{eqn.nearcorners}) and~(\ref{eqn.DiffOfMaxs}) that if we can find a vertex $v$ with $\X_v \in \SR \setminus (B_1 \cup B_3)$, and we set $u=\far(v)$ and $u'=\far(u)$, then $\X_{u}$ is in one of the balls $B(c_2,\omega)$ and $B(c_4,\omega)$ and $\X_{u'}$ is in the other.  It remains then to find such a vertex $v$.

For a measurable subset $S$ of $\SR$, let $N(S)$ be the random number of vertices~$v$ with $\X_v \in S$.  Observe that $N(B(c_i,s_n(k))) \sim \Bin(n,k/n)$, with mean~$k$.  For $i=1,3$,
\[ B_i \subseteq B(v_i, r+\omega) \subseteq B(c_i, r +2 \omega),\]
and whp (by Chebyshev's inequality),
\[N(B(c_i, r +2 \omega)) = (\pi/4)(r+2\omega)^2 + O(\omega \sqrt{n}) = (\pi/4)r^2 + O(\omega \sqrt{n}), \]
so whp each vertex $v$ with $\X_v \in B_i$ has rank $k(v_i,v) \leq (\pi/4)r^2 + \omega^2 \sqrt{n}$. 
Also, for $i=1,3$,
\[ B(c_i, r+ n^{1/3}) \subseteq B(v_i, r+ n^{1/3} +\omega) \subseteq B_i^+:= B(c_i, r+ n^{1/3} +2 \omega),\]
and whp
\[N(B(c_i, r +n^{1/3})) = (\pi/4)(r+n^{1/3})^2 + O(\omega \sqrt{n}) = (\pi/4)r^2 + \Theta(n^{5/6}), \]
so whp each vertex $v$ with $\X_v \not\in B_i^+$ has rank $k(v_i,v) \geq (\pi/4)r^2 + n^{4/5}$.  But the area of $\SR \setminus (B_1^+ \cup B_3^+)$ is $\Omega(n)$, so whp there are (many) vertices $v$ with
$\X_v \not\in (B_1^+ \cup B_3^+)$ and so with
$\min\{k(v_1,v), k(v_3,v)\} \geq (\pi/4)r^2 + n^{4/5}$.  Hence, by reading through $\tau_{v_1}$ and $\tau_{v_3}$ we can find a vertex $v$ with $\min\{k(v_1,v), k(v_3,v)\} \geq (\pi/4)r^2 + n^{4/5}$; and then
$\X_v \in \SR \setminus (B_1 \cup B_3)$,
and we are done. Note that we have looked at only four of the orders $\tau_v$.

There is an (unknown) symmetry $\pi$ of $\SR$ such that 
 $\X_{v_i}  \in B(\pi(c_i), \omega)$ for each $i \in [4]$;
and it follows that $\mathcal{C}_1$ holds whp, as required in step~(a).

\subsection{Preliminary results needed to analyse step (b)}
\label{subsec.b}

In this subsection we give preliminary results concerning (i) the area function $\lambda(s)$ for $\SU$, (ii) the
angle $c_2 \x c_4$ at a point $\x \in \SU$ far from $c_1$, (iii) estimating values $s_n(k)$ from ranks $k$, (iv) estimating Euclidean distances from values $s_n(k)$, (v) bounds on the ranks $k(v_i, v)$ for $i \neq i_0(v)$, and (vi) how to deal with the case when $k(v_{i_0 +2}, v)$ is large.
\medskip

\noindent
\emph{(i) On the area function $\lambda(s)$ for $\SU$}

If $0 \leq s \leq 1$ then clearly $\lambda(s) = \tfrac14 \pi s^2$.  
Let $1 < s < \sqrt{2}$. 
 Let $A$ be the point on the right side of $\SU$ at distance $s$ from the corner point $c_1$, so $A=(\frac12, -\frac12 +\sqrt{s^2 -1})$; and similarly $B=(-\frac12 +\sqrt{s^2 -1}, \frac12)$ 
is the point on the top side of $\SU$ at distance $s$ from $c_1$.
Let $\psi=\psi(s)$ be the angle $Ac_1B$, which is the angle subtended at $c_1$ by the curved part of the boundary of $\SU \cap B(c_1,s)$.  We claim that
\begin{equation} \label{eqn.phi}
  \psi(s) = \sin^{-1} (2s^{-2} -1) 
\end{equation} and
\begin{equation} \label{eqn.lambda}
 \lambda(s) = \tfrac12 s^2 \psi(s) + \sqrt{s^2 -1}.
\end{equation}
To establish this claim, let $\theta$ be the angle $Ac_1c_4$.  Then $\cos \theta =1/s$ and so $\cos(2\theta) = 2 \cos^2 \theta -1= 2s^{-2} -1$.  But the angle $Bc_1c_2$ also equals $\theta$, so $\psi + 2 \theta = \pi/2$. Thus $\sin \psi = 2s^{-2} -1$, giving the formula for $\psi$ in~(\ref{eqn.phi}).
Also, the sum of the areas of the triangles $c_1Ac_4$ and $c_1Bc_2$ is $\sqrt{s^2-1}$, and the sector with straight sides $c_1A$ and $c_1B$ (and internal angle $\psi$) has area $\frac12 s^2 \psi$; and these add up to $\lambda(s)$, establishing~(\ref{eqn.lambda}).

For $s = \tfrac{2}{\sqrt{3}}$ we have $\psi = \sin^{-1} \tfrac12= \tfrac{\pi}{6}$; and so
\begin{equation} \label{eqn.lambda3}
 \lambda(s) = \tfrac12 \cdot \tfrac43 \cdot \tfrac{\pi}{6} + \tfrac{1}{\sqrt{3}}= \tfrac{\pi}{9} + \tfrac{1}{\sqrt{3}} = \alpha_0  \approx 0.926416 .
\end{equation}
Also, for $1 \leq s \leq \tfrac{2}{\sqrt{3}}$ we have 
$\psi \geq \tfrac{\pi}{6}$.
\medskip

\noindent
\emph{(ii) On the angle $c_2 \x c_4$ at a point $\x \in \SU$ far from $c_1$}

We need to consider values of $s$ near to $\sqrt{2}$.  We shall show that
\begin{equation} \label{claim.psi}
\mbox{for each } \x \in \SU \mbox{ with } \dE(c_1,\x) \geq \tfrac{2}{\sqrt{3}}, \mbox{ the angle $c_2 \x c_4$ is at most } \tfrac{2 \pi}{3}.
\end{equation}
(We shall use this result in the proof of Lemma~\ref{lem.squarelike}.)
To prove~(\ref{claim.psi}), let $F$ (for `far' from $c_1$) be the set of points $(x,y) \in \SU$
with $x+y \geq \tfrac1{\sqrt{3}}$. 
The line $x+y= \frac1{\sqrt{3}}$ meets the line $x=\frac12$ at the point $P = (\frac12, \frac1{\sqrt{3}}-\frac12)$, and meets the line $y=\frac12$ at the point $Q = (\frac1{\sqrt{3}}-\frac12, \frac12)$. 
Consider the midpoint $\x^* =  (\tfrac1{2\sqrt{3}}, \tfrac1{2\sqrt{3}})$ of the segment of the line $x+y= \tfrac1{\sqrt{3}}$ between $P$ and $Q$.
The point  $\x^*$ is at distance $\tfrac1{\sqrt{6}}$ from the origin $O$.  Thus the angle $c_2 \x^* O$ 
is $\tan^{-1} \tfrac{1/\sqrt{2}}{1/\sqrt{6}}= \tan^{-1} \sqrt{3} = \pi/3$; and so the angle $c_2 \x^* c_4$ is $2\pi/3$.
We claim that, 
\begin{equation} \label{claim.2pi3}
\mbox{for each point } \x \in F, \mbox{ the angle } c_2 \x c_4 \mbox{ is at most } \tfrac{2\pi}{3}.
\end{equation}
This will follow from the above, once we check that the angle is maximised over $\x \in F$ at $\x = \x^*$.  
Clearly it is maximised at some point on the line $x+y= \tfrac1{\sqrt{3}}$.
Note that the lines $c_2c_4$ and $x+y=\frac1{\sqrt{3}}$ are parallel.  Let $a,b>0$, and consider the parallel lines $y=0$ and $y=b$.  Consider the origin O and the points $C=(a,0)$
 on the line $y=0$.  For each point $\z=(z,b)$ on the line $y=b$, let $\theta(z)$ be the angle $O\z C$.
It suffices now to show that $\theta(z)$ is maximised at $z=a/2$.
Write $\theta(z)$ as $\tan^{-1} \tfrac{z}{b} + \tan^{-1} \tfrac{a-z}{b}$, and differentiate: we find
\begin{eqnarray*}
\theta'(z) & = &
 \frac1{1+(z/b)^2} \frac1{b} \: + \: \frac1{1+((a-z)/b)^2} \big(\!-\!\frac1{b} \big)\\
 &=&
 \frac{ab}{(b^2+z^2)(b^2+(a-z)^2)} (a-2z)
\end{eqnarray*}
after some simplification.  Thus indeed $\theta(z)$ is maximised at $z=a/2$; and we have established the claim~(\ref{claim.2pi3}).

Since $P$ and $Q$ lie  on the line $x+y=\frac{1}{\sqrt{3}}$ and
\[ \dE(c_1,P) = \dE(c_1,Q) = (1+ \tfrac13)^{1/2} = \tfrac{2}{\sqrt{3}} \;\; (\approx 1.1547),
\]
 it follows that $\SU \setminus B(c_1, \tfrac{2}{\sqrt{3}}) \subseteq F$.  This completes the proof of~(\ref{claim.psi}).
\medskip

\noindent
\emph{(iii) Estimating values $s_n(k)$ from ranks $k$}

We now check that we can quickly estimate the values $s_n(k)$ needed to determine radii of circles and annuli.
\begin{lemma} \label{lem.estimate-s}
For fixed $0 < c_0 < c_1<1$, in linear time we can estimate all the values $s_n(k)$ for $c_0 n \leq k \leq c_1 n$ up to an additive error $\pm 1$.
\end{lemma}
\begin{proof}
Let $I$ be the set of integers $k$ with $c_0 n \leq k \leq c_1 n$.  Let $t = \lceil n^{1/3} \rceil$ (the exact value is not critical), let $J$ be the set of integers $k$ with $c_0 n - t \leq k \leq c_1 n$ such that $t\, | \, k$, and note that $|J| = \Theta(n^{2/3})$.
For any given value $k$, by repeated bisection (or another method), we can calculate an estimate $\hat{s}_n(k)$ of $s_n(k)$ up to accuracy $\pm \frac12$ in time $O(\log n)$.  
Thus we can calculate all the estimates $\hat{s}_n(k)$ for $k \in J$ in time $o(n)$. 
 For $k \in I \setminus J$, set the estimate $\hat{s}_n(k)$ to be $\hat{s}_n(j)$ where $j$ is the largest element of $J$ which is at most $k$.
We can calculate all these estimates in time $O(n)$.

It remains to check that these latter estimates for $i \in I \setminus J$ are accurate up to an additive error $\pm 1$.
Recall that, for $0 \leq s \leq \sqrt{2n}$, we let $\lambda_n(s)= \lambda(s/\sqrt{n})\, n$, the area of $B(c_1,s) \cap \SR$. 
Thus $\lambda_n(s_n(k))=k$.  
For each $k \in I$, the curved part of the boundary of $B(c_1,s_n(k)) \cap \SR$ has length $\Theta(\sqrt{n})$.
But if the curved part of the boundary of $B(c_1,s) \cap \SR$ has length $\ell$ then
\begin{equation} \label{eqn.curvedpart} \lambda_n(s+\tfrac12) - \lambda_n(s) > \tfrac12 \ell.
\end{equation}
Let us see why this is true. For $r>0$ and $0 \leq \alpha < 2\pi$, let $\mu(r,\alpha)$ be the area of a sector of radius $r$ and central angle $\alpha$, so $\mu(r,\alpha)= \frac12 \alpha r^2$. Note that the curved part of the boundary of such a sector has length $\alpha r$.  Let $P$ and $Q$ be the points where the curved part of the boundary of $B(c_1,s\!+\!\frac12) \cap \SR$ meets the sides of the square $\SR$, and let $\beta$ be the angle $Pc_1Q$. Then
\[ \lambda_n(s+\tfrac12) - \lambda_n(s) \geq \mu(s+\tfrac12,\beta) - \mu(s,\beta)= \tfrac12 \beta((s+\tfrac12)^2-s^2) = \tfrac12 \beta(s+\tfrac14) > \tfrac12 \ell, \]
proving~(\ref{eqn.curvedpart}).


By~(\ref{eqn.curvedpart}) we have
\[ \lambda_n(s_n(k)+ \tfrac12) = k + \Theta(\sqrt{n}) > k+t =\lambda_n(s_n(k+t)), \]
and so
\[ s_n(k) < s_n(k+t) < s_n(k)+ \tfrac12.\]
Now let $i \in I \setminus J$ and let $j$ be the largest element of $J$ which is at most $i$ (so $\hat{s}_n(i) = \hat{s}_n(j)$).  Then
\[ | \hat{s}_n(i) - s_n(i) |
\leq | \hat{s}_n(j) - s_n(j) | + | s_n(j) - s_n(i)| < 1.
\]
Thus the estimates for $i \in I \setminus J$ are accurate up to an additive error $\pm 1$, as required.
\end{proof}

\medskip

\noindent
\emph{(iv) Estimating Euclidean distances from values $s_n(k)$}
We have seen that in linear time we can estimate the values $s_n(k)$ up to an
additive error $\pm 1$. We can use these values to estimate Euclidean distances
from corners.
Recall that $\alpha_0 = \lambda(\tfrac{2}{\sqrt{3}}) \approx 0.9264$, see~(\ref{eqn.lambda3}).
%
\begin{lemma} \label{lem.conc}

Assume that $\mathcal{C}_1(v_1,\ldots,v_4)$ holds. There exists $\eps>0$ such that if $\alpha = \alpha_0+\eps$ then the following holds whp.
For each $i \in [4]$ and each $v \in V^-$, 
if $k=k(v_i,v)$ satisfies $k =\Omega(n)$ and $k \leq \alpha n$, then
\begin{equation} \label{eqn.conc}
\big| \dE(\pi(c_{i}), \X_v) - s_n(k) \big| \leq 1.19695 \sqrt{\log n}.
\end{equation}
\end{lemma}
%
\begin{proof}
Let $i \in [4]$.  Let $k$ be an integer with $k=\Omega(n)$ and $k \leq \alpha n$.
Observe that $N(B(c_i,s_n(k))) \sim \Bin(n,k/n)$, with mean $k$.
Let $0<\eta<1$: later we shall insist that $\eta$ is sufficiently small that a certain inequality holds.
By~(\ref{eqn.lambda3}), by taking $\eps$ sufficiently small, we may ensure that $s(k/n) \leq (1+\eta) \tfrac{2}{\sqrt{3}}$ and the angle $\psi = \psi(s(k/n))$ satisfies $\psi \geq \psi_0$,
where $\psi_0 = (1-\eta) \frac{\pi}{6}$.
Now, for a given constant $c>0$,
\begin{eqnarray*}
&&
\lambda_n (s_n(k)+c \sqrt{\log n}) - \lambda_n(s_n(k)) \\
& \geq & (1+o(1)) \, \psi_0 \,\big( (s_n(k) +c\sqrt{\log n})^2 - s_n(k)^2 \big)\\
& =  &
(1+o(1))\, 2c \,\psi_0\, s_n(k) \sqrt{\log n}.
\end{eqnarray*}
Also, since $\psi \leq \tfrac{\pi}{2}$, 
\begin{eqnarray*}
&&
\lambda_n(s_n(k) +c \sqrt{\log n}) - \lambda_n(s_n(k))\\ 
& \leq &
(1+o(1)) \,\tfrac{\pi}{2}\, \big( (s_n(k) + c \sqrt{\log n})^2 - s_n(k)^2 \big)\\
& \leq &
(1+o(1)) \pi\, c \, s_n(k) \sqrt{\log n},
\end{eqnarray*}
so
\[ 1 \leq \lambda_n(s_n(k)+c \sqrt{\log n}) / \lambda_n(s_n(k)) \leq 1+ O\big(\sqrt{\tfrac{\log n}{n}}\big)= 1+o(1).
\]

Let $X^+ = N(B(c_i,s_n(k) +c\sqrt{\log n}))$.  By Lemma~\ref{lem:Chernoff}, 
since $s_n(k) \sqrt{\log n} / \E[X^+] = o(1)$,
\begin{eqnarray*}
\pr(X^+ \leq k) & = &
\pr(X^+ \leq \E[X^+] (1-(1+o(1))\, 2c \,\psi_0 \, s_n(k) \sqrt{\log n}/\E[X^+])\\
& \leq &
\exp \{ -(1+o(1))\, \tfrac12 \, \big( 2c\,\psi_0 \, s_n(k) \sqrt{\log n} / \E[X^+] \big)^2  \big) \E[X^+] \}\\
& \leq &  
\exp \{ -(1+o(1))\, 2\, c^2 \psi_0^2\, s_n(k)^2 \log n / \E[X^+] 
\}.
\end{eqnarray*}
But
\[ \E[X^+] = (1+o(1))\, k = \lambda_n(s_n(k)) \leq \tfrac14 \pi\, s_n(k)^2,\]
so
\begin{eqnarray*}
 \pr(X^+ \leq k) & \leq & 
 \exp\big(  -(1+o(1)) \frac{\tfrac14 \pi s_n(k)^2}{\E[X^+]} \frac{8}{\pi} \, c^2 \psi_0^2 \log n \big)\\
 & \leq &
\exp\big(  -(1+o(1)) \tfrac{2 \pi}{9} c^2 (1-\eta)^2 \log n \big).
\end{eqnarray*}
Note that $\sqrt{9/(2\pi)} \approx 1.196827$.
Set $c=1.1969$, and choose $\eta>0$ sufficiently small that $\tfrac{2 \pi}{9} c^2 (1-\eta)^2 >1$.
Now we have $\pr(X^+ \leq k)= o(1/n)$.

Similarly, let $X^- = N(B(c_i, s_n(k) -c \sqrt{\log n}))$: then, with the same value of~$c$,
\[ \pr(X^- \geq k) = o(1/n).\]  Thus we have seen that whp the following holds.
For each $i \in [4]$ and each $v \in V^-$, 
if $k=k(v_i,v)$ satisfies $k=\Omega(n)$ and $k \leq \alpha n$, then
\[ N(B(c_i,s_n(k) - 1.1969 \sqrt{\log n})) <k  \; \mbox{ and } \; N(B(c_i,s_n(k) + 1.1969 \sqrt{\log n})) >k.\]
Now, for each $i \in [4]$, $d_E(\pi(c_i), v_i) < \omega \ll \sqrt{\log n}$.
Also
\begin{eqnarray*}
B(\pi(c_i),s_n(k) + 1.1969 \sqrt{\log n})
& \subseteq &
B(v_i, s_n(k) + 1.1969 \sqrt{\log n} + \omega)\\
& \subseteq &
B(\pi(c_i), s_n(k) + 1.1969 \sqrt{\log n} + 2\omega)\\
& \subseteq &
B(\pi(c_i), s_n(k) + 1.19695 \sqrt{\log n})
\end{eqnarray*}
(for $n$ sufficiently large).
But the first of these four balls contains more than $k$ points $\X_u$, so $\X_v$ must be in the second ball, and so it is in the last one; that is
$d_E(\pi(c_i),\X_v) \leq s_n(k) + 1.19695 \sqrt{\log n}$.
  Similarly, $\X_v \not\in B(\pi(c_i), s_n(k) - 1.19695 \sqrt{\log n})$,
 and the lemma follows.
\end{proof}


\medskip


\noindent
\emph{(v) Bounding the ranks $k(v_i,v)$ for $i \neq i_0(v)$}

Condition throughout on the event $\mathcal{C}_1$, and on a particular choice of $v_1,\ldots,v_4$; that is, condition on the event $\mathcal{C}_1(v_1,\ldots,v_4)$. Let $V^- = V \setminus \{v_1,\ldots,v_4\}$.
 Recall that, for each $i \in [4]$ and $v \in V^-$, $k(v_i,v)$ is the rank of $v$ in the order $\tau_{v_i}$.  
 Since $v_i$ is very close to $\pi(c_i)$ whp, we may think of $k(v_i,v)$ as roughly the number of points $\X_u$ for $u \in V$ which are as close to $\pi(c_i)$ as $\X_v$ is.
Let $\mathcal{C}_5$ 
be the event that, for each $j \in [4]$, 
\[ \big| \{ u \in V^{-} : \dE(c_j, \X_u) < \tfrac12 \sqrt{n} - 2\omega \}\big| \geq \tfrac{\pi}{16} n - n^{2/3}.\]
Then $\mathcal{C}_5$ holds whp, by Chebyshev's inequality.  
Let $i \in[4]$ and let $v \in V^-$.  If $d_E(\pi(c_i), \X_v) \geq \frac12 \sqrt{n}$ then
$d_E(\X_{v_i}, \X_v) \geq \frac12 \sqrt{n} - \omega$, and so
each vertex $u$ such that $\dE(\pi(c_i), \X_u) < \tfrac12 \sqrt{n} - 2\omega$ satisfies $\dE(\X_{v_i}, \X_u) < \dE(\X_{v_i}, \X_v)$; hence, if $d_E(\pi(c_i), \X_v) \geq \frac12 \sqrt{n}$ and $\mathcal{C}_5$ holds, then
$k(v_i,v) > \frac{\pi}{16}n - n^{2/3}$.

Recall that, given $v \in V^-$, the index $i_0= i_0(v) \in [4]$ satisfies $k(v_{i_0},v) = \min_{i \in [4]} k(v_i,v)$ (breaking ties by choosing the least such value $i$).
Condition on $\mathcal{C}_5$ holding. Then, for each $v \in V^-$ and each $i \in [4] \backslash \{i_0\}$, we have $k(v_i,v) > \tfrac{\pi}{16} n - n^{2/3}$.  (For, if not, then both $\dE(\pi(c_{i_0}),\X_v) < \tfrac12 \sqrt{n}$ and 
$\dE(\pi(c_{i}),\X_v) < \tfrac12 \sqrt{n}$, which is not possible since the distance between distinct corners is at least $\sqrt{n}$.)
Note that $\pi/16 \approx 0.1963 > 0.19$.
 Hence, for each $i \in [4] \backslash \{ i_0\}$ we have $k(v_i,v)  > 0.19 \, n$.
\smallskip

Next, we show that for $i = i_0 \pm 1$ (indices are taken modulo $4$), we have $k(v_i,v) \le \alpha n$.  Assume wlog that $i_0=1$, and consider $i=2$.
 We saw earlier that $\X_v$ is within distance $\omega$ of the quarter square containing $\pi(c_1)$.
   Recall that $v_2$ is close to the corner $\pi(c_2)$. The maximum distance from $\pi(c_2)$ to $\X_v$ is $(1+o(1))\sqrt{5n}/2$.
But $\lambda(\sqrt{5}/2) = \frac58 \sin^{-1} \frac35 + \frac12 \approx 0.902188$.  Thus the area of $B(\pi(c_2), (1+o(1))\sqrt{5n}/2)$ is $< 0.905 n$.  Hence, by Lemma~\ref{lem:Chernoff}, 
wvhp the number of vertices $w$ with $\X_w \in B(\pi(c_2), (1+o(1)) \sqrt{5n} / 2)$ is less than $0.91n$, so $k(v_2, v) < 0.91 n < \alpha n$, as required.
\medskip

\noindent
\emph{(vi) The case when $k(v_{i_0+2},v)$ is large}

For $i=i_0+2$, we have $k(v_i,v) > 0.19 n$, but the upper bound $k \le \alpha n$ might or might not hold. In order to deal with both cases, we need another auxiliary lemma. 
\begin{lemma} \label{lem.squarelike}
For $v \in V^-$, let $i_0=i_0(v)$. Then the following holds whp. For each $v \in V^-$ with $k(v_{i_0 +2},v) > \alpha n$, the near-rhombus formed from the intersection of the two  
annuli centred on the corners $\pi(c_{i_0-1})$ and $\pi(c_{i_0+1})$ is squarelike, i.e., the angles in the near-rhombus are between $\pi/3$ and $2\pi/3$.
\end{lemma}
\begin{proof}
Let $A_1$ be the event that $v_i \in B(\pi(c_i), \omega)$ for each $i \in [4]$.
Then $A_1$ holds whp.
Let $A_2$ be the event that 
$N(B(c_i, \tfrac{2}{\sqrt{3}} \sqrt{n} + \omega)) < \alpha n$ for each $i \in [4]$.
Then $A_2$ holds wvhp by Lemma~\ref{lem:Chernoff}, since by~\eqref{eqn.lambda3} 
the area of $B(\pi(c_{i_0+2}), \frac{2}{\sqrt{3}}\sqrt{n} + \omega)$ is $(1+o(1)) \alpha_0 n$, and $\alpha>\alpha_0$.

For $v \in V^-$, let $A_3(v)$ be the event that $k(v_{i_0+2},v) > \alpha n$, and let $A_4(v)$ be the event that $d_E(\pi(c_{i_0+2}), \X_v) \geq \frac{2}{\sqrt{3}} \sqrt{n}$.
If $A_4(v)$ holds, then, by~\eqref{claim.psi}, the angle $c_{i_0-1} \X_v c_{i_0+1}$ is at most $2 \pi/3$ (and clearly at least $\pi/2$). Hence the intersection of the two 
annuli centred on the corners $\pi(c_{i_0-1})$ and $\pi(c_{i_0+1})$ forms a near-rhombus such that the angles are between $\pi/3$ and $2\pi/3$, that is, it is squarelike.
Thus we want to show that whp, for each  $v \in V^-$, if $A_3(v)$ holds then $A_4(v)$ holds.

But on $A_1 \land A_2$, for each  $v \in V^-$, if $A_4(v)$ fails then
\begin{eqnarray*}
k(v_{i_0+2},v) &=&
\big| \{u \in V: d_E(v_{i_0+2}, \X_u) \leq d_E(v_{i_0+2}, \X_v) \} \big|\\
& \leq &
N(B(v_{i_0+2}, \tfrac{2}{\sqrt{3}} \sqrt{n}))\\
& \leq &
N(B(\pi(c_{i_0+2}), \tfrac{2}{\sqrt{3}} \sqrt{n} + \omega)) \;\; < \;
\alpha n
\end{eqnarray*}
so $A_3(v)$ fails.  In other words, on $A_1 \land A_2$, for each  $v \in V^-$, if $A_3(v)$ holds then $A_4(v)$ holds; and since  $A_1 \land A_2$ holds whp, this completes the proof.

\end{proof}
\smallskip

\subsection{Completing the proof of Theorem~\ref{thm:main2}}

In order to finish the proof of Theorem~\ref{thm:main2},
we may assume wlog that, in step (a), in $O(n)$ time we have found `corner' vertices $v_1,\ldots,v_4$ such that $\mathcal{C}_1(v_1,\ldots,v_4)$ holds; and we may assume wlog that the random permutation $\pi$ is the identity map (as in the proof of Theorem~\ref{thm:main}). For each vertex $v \in V^-$, we form the rank list $R(v)=(k(v_i,v):i=1,\ldots,4)$.
We may do this in linear time, by reading through the four orders $\tau_{v_i}$. 

Recall that whp, for each vertex $v \in V^-$,  $k(v_i,v) > 0.19\,n$ for each $i \neq i_0$,
and  $k(v_i,v) < 0.91 \,n$ for $i=i_0 \pm 1$. Assume wlog that these inequalities hold, and consider a vertex $v \in V^-$.  
We distinguish the two cases, whether $k(v_{i_0+2},v) \le \alpha n$ or not.
\smallskip

\textit{Case 1: $k(v_{i_0+2}, v) \le \alpha n$.} \\
In this case, the ideas of Theorem~\ref{thm:main} can be applied.
Let $I^-= [4] \backslash \{i_0\}$. 
By Lemmas~\ref{lem.estimate-s} and~\ref{lem.conc}, whp, for each vertex $v$ and each $i \in I^-$, we can calculate the value $d_E(c_i, \X_v)$ up to an additive error of $1.19695 \sqrt{\log n} +1$.
%
Now, exactly as in the proof of Theorem~\ref{thm:main}, we consider three circles $C_i(v)$ (with corresponding annuli $A_i(v)$) for $i \in I^-$,
and pick a pair of circles 
meeting at an angle between $\pi/3$ and $2\pi/3$.
We set $\Phi(v)$ to be a point we compute which is within distance 1 of (or arbitrarily close to) the relevant point where these circles 
meet, 
 and then $d_E(\Phi(v),\X_v) < 1.197 \sqrt{\log n}\,$ (for $n$ sufficiently large).  
 \smallskip
 
 \textit{Case 2: $k(v_{i_0+2}, v) > \alpha n$.} \\
As in the last case, by Lemmas~\ref{lem.estimate-s} and~\ref{lem.conc}, we can calculate the values $d_E(c_{i_0-1}, \X_v)$ and $d_E(c_{i_0+1}, \X_v)$ up to an additive error of $1.19695 \sqrt{\log n} +1$.
 In this case, by Lemma~\ref{lem.squarelike}, the two circles (with corresponding annuli) 
centred on the corners $c_{i_0-1}$ and $c_{i_0+1}$ meet at an angle between $\pi/3$ and $2\pi/3$.
As before, we set $\Phi(v)$ to be a point within distance 1 of
the relevant point where these circles 
meet, and we find that $d_E(\Phi(v),\X_v) < 1.197 \sqrt{\log n}$.

In both cases, all the calculations can be completed in linear time.

\section{Concluding remarks}\label{sec:Conclusion}

Recall that there is a family of $n$ random points $\X_v$ for $v \in V$, independently and uniformly distributed in the square $\SR = \left[-\sqrt{n}/2,\sqrt{n}/2\right]^2$ of area $n$.  We do not see these points, but learn about them in one of the following two ways: 
(a) when we are given just the corresponding random geometric graph (for a suitable threshold distance $r$), and (b) when we have some geometric information.
In case (a), 
we obtained an embedding $\Phi$ with displacement at most about $r$, but we required the threshold distance $r$ to satisfy $r \gg n^{3/14}$, which yields rather a dense random geometric graph.
In case (b), for each vertex $v$, we are given a list of all the vertices $w$ ordered by increasing Euclidean distance from $\X_v$ of the corresponding points $\X_w$.
In this case, we obtain an embedding $\Phi$ with dramatically less error. 

Can we obtain lower displacement 
 for these approximate reconstruction
 problems?  Can we obtain similar low displacement
for smaller values of $r$ (yielding sparser random graphs)? 
Another open issue is whether there is a different choice of non-trivial natural geometrical information that would help to extend the range of values of $r$ we can handle. Notice that exposing the real length of all the edges would trivialise the problem, as we saw in Subsection~\ref{subs:further}.
Another natural line of research is to consider a region in the plane different from the square $\SR$, for instance a disk of area $n$, still with $n$ iid uniformly distributed random points $\X_v$ yielding a random geometric graph $G$.   Here we cannot of course start 
from the corners,
but we do have a boundary and we can identify vertices $v$ of $G$ with $\X_v$ near the boundary by looking at vertex degrees.  Similarly, it would be interesting to investigate such problems with distributions different from the uniform distribution.
 
Also it would be interesting to generalise the problem to higher dimensions, to $\real^d$ for $d\ge 3$. We believe that for bounded dimension $d$, or indeed for sufficiently slowly growing dimension, similar results to those obtained in this paper could be obtained for $n$ iid points uniformly distributed in the $d$-cube $\left[-n^{1/d}/2, n^{1/d}/2\right]^d$ of volume $n$.

Finally, let us mention
the model where the underlying space is the unit sphere ${\mathbb S}^{d-1}$ in $\R^d$
(with $n$ iid uniformly distributed random points $\X_v$).  See~\cite{bder16}
 for recent work on this model in high dimensions, where the main interest is to test whether we are looking at a random geometric graph from this model or at a corresponding Erd\H{o}s-R\'enyi random graph.  See also the references in~\cite{bder16} for other work on this model.
For the approximate realization or reconstruction problem, there is now not even a boundary to start from! 

\textbf{Acknowledgement.} We thank the referees for careful reading and helpful suggestions.

\end{document}